\documentclass[10pt]{article}

\usepackage[T1]{fontenc}
\usepackage{lmodern}
\usepackage{amsmath,amsthm,amssymb}
\usepackage{booktabs}
\usepackage{mathtools}
\usepackage{array}
\usepackage{color}
\usepackage{mathrsfs}
\usepackage{microtype}
\usepackage{needspace}
\usepackage[hidelinks]{hyperref}
\hypersetup{
 pdftitle={The Arithmetic Geometry of Square-Sided Heron Triangles},
 pdfauthor={Yangcheng Li}
}

\renewcommand{\arraystretch}{1.7}
\newcolumntype{L}[1]{>{\raggedright\arraybackslash}p{#1}}

\newtheorem{theorem}{Theorem}[section]
\newtheorem{proposition}[theorem]{Proposition}
\newtheorem{lemma}[theorem]{Lemma}
\newtheorem{corollary}[theorem]{Corollary}

\theoremstyle{definition}
\newtheorem{definition}[theorem]{Definition}
\newtheorem{example}[theorem]{Example}

\theoremstyle{remark}
\newtheorem{remark}[theorem]{Remark}

\newtheorem*{theoremA}{Theorem A}
\newtheorem*{theoremB}{Theorem B}
\newtheorem*{theoremC}{Theorem C}

\numberwithin{equation}{section}

\newcommand{\Q}{\mathbb{Q}}
\newcommand{\Z}{\mathbb{Z}}
\newcommand{\F}{\mathbb{F}}
\newcommand{\R}{\mathbb{R}}
\newcommand{\C}{\mathbb{C}}
\newcommand{\PP}{\mathbb{P}}
\newcommand{\abs}[1]{\lvert #1\rvert}
\newcommand{\ord}{\operatorname{ord}}
\newcommand{\rank}{\operatorname{rank}}
\newcommand{\Jac}{\operatorname{Jac}}
\newcommand{\Prym}{\operatorname{Prym}}
\newcommand{\Eadm}{\mathcal E_{b}^{\mathrm{adm}}}

\title{The Arithmetic Geometry of\\
Square-Sided Heron Triangles}

\author{Yangcheng Li$^{1,}$\footnote{E-mail: liyc@m.scnu.edu.cn. Supported by NSF of China No.12571003, 12501006; Basic and Applied Basic Research Foundation of Guangdong Province No. 2024A1515010589.}\\
{\small\it $^{1}$School of Mathematical Sciences, South China Normal University,}\\
{\small\it Guangzhou 510631, Guangdong, P. R. China}}
\date{}

\begin{document}
\baselineskip15pt
\maketitle

\begin{abstract}
We study rational Heron triangles with two marked square sides using
elliptic curves and K3 surfaces.  An explicit quartic-to-elliptic
correspondence parametrizes marked similarity classes by rational
points satisfying a positivity condition, modulo
\((x,y)\sim(x,-y)\).  We determine the generic Mordell--Weil group,
prove that every \(k\in\mathbf Q\setminus\{0,\pm1\}\) supports
infinitely many scalene classes with exactly two square sides, and
construct a primitive family with \(N(X)\gg X^{1/4}\).  Requiring the
third side to be square gives a genus-three Ciani quartic whose
Jacobian is \(\mathbf Q\)-isogenous to a product of three elliptic
curves.  Geometrically, the two constructions give inequivalent
elliptic fibrations on a single singular K3 surface, with geometric
Mordell--Weil ranks \(2\) and \(0\).  The minimal resolution of the
all-square locus is a surface of general type with invariants
\((K^2,p_g,q)=(2,3,0)\).  Assuming weak Bombieri--Lang, parameters
yielding a nondegenerate all-square triangle form a thin subset of
\(\mathbf P^1(\mathbf Q)\).
\end{abstract}

\noindent{\bf Keywords:} Heron triangle, square side, elliptic curve, Ciani quartic, K3 surface.

\noindent{\bf 2020 Mathematics Subject Classification:}
Primary 11G05, 14G05; Secondary 11D25, 14J28.

\section{Introduction and main results}\label{sec:introduction}

A \emph{rational Heron triangle} is a nondegenerate triangle whose
side lengths and area are positive rational numbers.  An
\emph{integral Heron triangle} has integral side lengths and integral
area; it is \emph{primitive} when the greatest common divisor of its
three side lengths is \(1\).  A rational side is called a
\emph{square side} when its length belongs to \(\Q^{\times 2}\).
For a prime \(p\), \(\ord_p\) denotes the normalized \(p\)-adic
valuation.

The requirement that two or three sides be squares looks elementary
but leads naturally to arithmetic geometry.  Sastry recorded early
examples with two square sides \cite{Sastry}; later searches produced
larger tables \cite{Lagneau}.  To the author's knowledge, the only
primitive all-square Heron triangles currently recorded are
\[
 (1853^2,4380^2,4427^2)
 \quad\text{and}\quad
 (11789^2,68104^2,68595^2);
\]
 see \cite{OEISSquareAreas,Rathbun,StanicaEtAl}.  Whether further
 primitive all-square Heron triangles exist remains open.

Related Heron-triangle problems with figurate-number side constraints
were studied by Peng and Zhang~\cite{PengZhang}.
The correspondence between Heron triangles of prescribed area and
elliptic curves was developed by Goins and Maddox
\cite{GoinsMaddox} and revisited by Halbeisen and Hungerb\"uhler
\cite{HalbeisenHungerbuehler}.  K3 geometry also appears in van
Luijk's construction from triangles with common area and perimeter
\cite{VanLuijk}; that K3 surface has Picard number \(18\).  In
contrast, the K3 surface arising below from square-side constraints
has geometric Picard number \(20\).

Elliptic fibrations and Mordell--Weil lattices on Kummer and Inose
surfaces attached to products of elliptic curves have been studied
systematically; see, for example, Kumar and Kuwata
\cite{KumarKuwata}.  We do not claim that the two fiber
configurations below are new in isolation.  The point here is that
these two specific, inequivalent fibrations arise naturally from the
two-square and third-square stages of the Heron problem and become
unified geometrically on one singular K3 surface.

The plane quartics arising from the third-square condition form a
special one-parameter subfamily of Ciani quartics.  In the projective
coordinates used below, our equation
\[
 Q^4=Z^4-2aZ^2W^2+W^4
\]
is precisely Gu\`ardia's family
\[
 Y^4=X^4-(4n-2)X^2Z^2+Z^4,
 \qquad n=\frac{a+1}{2}.
\]
Gu\`ardia constructed the geometric elliptic quotients of this family
and the resulting splitting of its Jacobian
\cite{GuardiaJLMS,GuardiaJTNB}.  Accordingly, the existence of these
quotients and the geometric splitting are not claimed here as new.
The contribution of the present work is their arithmetic realization
in the Heron-triangle parametrization.  The main results fall into three
parts: two-square arithmetic, including exact generic Mordell--Weil
groups, uniform fiberwise infinitude, primitive scaling, and counting;
three-square geometry, including the two contrasting elliptic
fibrations on one geometric singular K3 surface and generic
rational-point rigidity; and the global all-square surface of general
type together with a conditional thinness theorem.

The central device of this paper is an explicit quartic-to-elliptic
correspondence.  This framework separates three
logically distinct issues:
classification of marked rational similarity classes, existence of a
square-preserving primitive representative, and the extra quadratic
lifting condition imposed by a third square side.

Fix
\begin{equation}\label{eq:intro-ab}
 a,b\in\Q,\qquad a^2+b^2=1,\qquad b\ne0,
\end{equation}
and put
\[
 E_b:\quad y^2=x^3+4b^2x,\qquad
 \rho_{a,b}(x,y)=a+\frac{x}{4}-\frac{b^2}{x}.
\]
The following three statements organize the main results.

\begin{theoremA}[Two-square arithmetic]
Marked rational similarity classes with normalized sides
\((1,q^2,r)\) are in bijection with
\[
 \left\{(x,y)\in E_b(\Q):x\ne0,\ \rho_{a,b}(x,y)>0\right\}
 \big/\bigl((x,y)\sim(x,-y)\bigr),
\]
via
\[
 x=2(q^2+r-a),\quad y=4q(q^2+r-a),\qquad
 q=\frac{y}{2x},\quad r=\rho_{a,b}(x,y).
\]
Under
\[
 a=\frac{2k}{k^2+1},\qquad b=\frac{k^2-1}{k^2+1},
\]
the associated model
\[
 E_k:\quad Y^2=X^3+4(k^4-1)^2X
\]
has
\[
 E_k(\Q(k))=\Z Q_k\oplus\langle(0,0)\rangle
 \simeq\Z\oplus\Z/2\Z.
\]
Every \(k\in\Q\setminus\{0,\pm1\}\) supports infinitely many pairwise
nonsimilar scalene classes with exactly two square sides, and each can
be scaled by a rational square to an integral Heron triangle preserving
those two integer-square sides.

If a marked class is written in lowest terms as
\[
 \left(1,\left(\frac uv\right)^2,\frac wz\right),
 \qquad \gcd(u,v)=\gcd(w,z)=1,
\]
then it has a unique square-preserving primitive integral
representative if and only if
\[
 \frac{z}{\gcd(z,v^2)}
\]
is an integer square.  There is also an explicit primitive scalene
family with exactly two square sides for which the number \(N(X)\) of
unmarked members having largest side at most \(X\) satisfies
\(N(X)\gg X^{1/4}\).
\end{theoremA}

\begin{theoremB}[Three-square geometry]
Every rational all-square similarity class has a unique primitive
integral all-square representative, up to edge order.  For fixed
\(a,b\), the third-square condition gives the smooth Ciani quartic
\[
 C_{a,b}:\quad q^4=z^4-2az^2+1
\]
of genus \(3\), and its three elliptic quotients give
\[
 \Jac(C_{a,b})\sim_{\Q}E_b^2\times F_a,\qquad
 F_a:\quad V^2=2(U-a)(U-1)(U+1).
\]
Let \(K_0=\Q(k)\) and \(K_{\mathrm{geom}}=\overline{\Q}(k)\).
The two elliptic surfaces attached to \(E_k\) and \(F_a\) are
geometrically isomorphic singular K3 surfaces, but their displayed
fibrations are inequivalent.  The first is
\(\operatorname{Km}(E_0\times E_0)\), where \(E_0:y^2=x^3+x\), and has
geometric Mordell--Weil rank \(2\); the second is extremal of geometric
rank \(0\), with
\[
 F_a(K_{\mathrm{geom}})=F_a(K_0)
 =\{O,(a,0),(1,0),(-1,0)\}.
\]
The generic quartic has only its four boundary \(K_0\)-points and hence
no nontrivial rational section, while its Jacobian has arithmetic and
geometric generic ranks \(2\) and \(4\).
\end{theoremB}

\begin{theoremC}[Global all-square geometry]
The global all-square locus is the weighted double plane
\[
\begin{split}
\mathscr H:\quad \Omega^2={}&
(X^2+Y^2+Z^2)(-X^2+Y^2+Z^2)\\
&{}\times(X^2-Y^2+Z^2)(X^2+Y^2-Z^2)
\subset\PP(1,1,1,4).
\end{split}
\]
Geometrically, \(\mathscr H\) has exactly twelve Du Val singularities
of type \(A_3\).
Its minimal resolution \(\widetilde{\mathscr H}\) is a minimal surface
of general type with
\[
 (K^2,p_g,q)=(2,3,0).
\]
Assuming weak Bombieri--Lang for \(\widetilde{\mathscr H}\), the
rational parameters whose all-square fiber contains a nonboundary
rational point form a thin subset of \(\PP^1(\Q)\).
\end{theoremC}

Section~\ref{sec:normalization} develops marked normalization and the
primitive criterion.  Section~\ref{sec:elliptic-model} proves the
natural elliptic correspondence and records every exceptional point.
Section~\ref{sec:sections} determines the exact generic
Mordell--Weil group and proves uniform non-torsion, fixed-fiber
infinitude, the Kummer--K3 interpretation, and the isosceles
classification.
Section~\ref{sec:primitive-family} constructs and counts primitive
scalene examples.  Section~\ref{sec:genus-three} treats the all-square
genus-three cover, its quotients, the two elliptic fibrations on one
geometric singular K3 surface, generic endomorphisms and rigidity, and
local obstructions.  Section~\ref{sec:global-all-square-surface}
studies the global surface and conditional sparsity.
Section~\ref{sec:examples} records the known examples and fiberwise
arithmetic, and Section~\ref{sec:questions} states the principal open
problems.

\section{Marked similarity classes and primitive scaling}
\label{sec:normalization}

\begin{definition}
A \emph{marked two-square triangle} is a rational Heron triangle
together with an ordered choice of two of its square sides.  Two such
triangles are \emph{marked-similar} when an ordinary similarity sends
the first marked side to the first and the second to the second.
\end{definition}

If the first marked sides of two marked-similar triangles are
\(u^2\) and \(u'^2\), their similarity ratio is
\((u'/u)^2\).  Thus marked similarity automatically uses a rational
square scaling factor; there is no separate ambiguity about whether
squares are preserved.

\begin{proposition}[Marked normal form]\label{prop:marked-normal-form}
Every marked similarity class of nondegenerate rational Heron
triangles with at least two square sides has a representative
\[
 (1,q^2,r),\qquad q,r\in\Q,\quad q\ne0,\quad r>0.
\]
There exist \(a,b\in\Q\) satisfying \eqref{eq:intro-ab} such that
\begin{equation}\label{eq:normal-form}
 q^4=r^2-2ar+1.
\end{equation}
Conversely, rational numbers satisfying these conditions determine a
rational Heron triangle.
\end{proposition}

\begin{proof}
Write the two marked sides as \(u^2,w^2\), the third side as \(c\),
and the area as \(K\in\Q_{>0}\).  Scaling lengths by \(u^{-2}\)
gives
\[
 1,\qquad q^2=\left(\frac wu\right)^2,\qquad r=\frac{c}{u^2}.
\]
Place the side of length \(r\) on the horizontal axis and take the
vertices
\[
 (0,0),\qquad (r,0),\qquad (a,b).
\]
The first normalized side has length \(1\), hence \(a^2+b^2=1\).
The normalized area is \(K/u^4\), so
\[
 \frac{\abs b\,r}{2}=\frac{K}{u^4}.
\]
It follows that \(b\in\Q\); the cosine rule gives
\[
 a=\frac{r^2+1-q^4}{2r}\in\Q.
\]
The remaining squared distance is
\[
 q^4=(r-a)^2+b^2=r^2-2ar+1.
\]
Nondegeneracy is exactly \(b\ne0\).

Conversely, the three displayed vertices have side lengths
\((r,1,q^2)\) and area \(\abs b r/2>0\).  Hence they give the required
rational Heron triangle.
\end{proof}

\begin{lemma}[Integral sides and rational area]
\label{lem:integer-sides-rational-area}
An integer-sided triangle with rational area has integer area.
\end{lemma}

\begin{proof}
Let the sides be \(A,B,C\), the area be \(\mathcal A\), and write
Heron's identity as
\[
 16\mathcal A^2
 =(A+B+C)(-A+B+C)(A-B+C)(A+B-C)=N\in\Z.
\]
Thus \(4\mathcal A=\sqrt N\in\Q\).  A rational square root of an
integer is an integer, so \(\sqrt N\in\Z\).

It remains to show that \(4\mid\sqrt N\).  If \(A+B+C\) were odd,
then either all three sides would be odd or exactly one would be odd.
Using
\[
 N=2A^2B^2+2A^2C^2+2B^2C^2-A^4-B^4-C^4,
\]
these cases give \(N\equiv3\pmod8\) and
\(N\equiv7\pmod8\), respectively, neither of which is a square.
Thus the perimeter is even.  Each of the four Heron factors is then
even, so \(16\mid N\).  Since \(N\) is a square,
\(4\mid\sqrt N\), and \(\mathcal A\in\Z\).
\end{proof}

The preceding lemma does not imply primitivity.  The obstruction is
local and can be stated exactly.

\begin{theorem}[Exact primitive square-scaling criterion]
\label{thm:primitive-scaling}
Suppose a marked rational class is written as
\[
 \left(1,\left(\frac uv\right)^2,\frac wz\right),
\]
where
\[
 u,v,w,z\in\Z_{>0},\qquad
 \gcd(u,v)=\gcd(w,z)=1.
\]
It possesses a primitive integral representative whose first two
sides remain integer squares if and only if
\begin{equation}\label{eq:primitive-criterion}
 {\frac{z}{\gcd(z,v^2)}=t^2
 \quad\text{for some }t\in\Z_{>0}.}
\end{equation}
When this holds, put
\begin{equation}\label{eq:minimal-square-scale}
 s=\prod_p p^{c_p},\qquad
 c_p=\max\left\{\ord_p(v),
                 \left\lceil\frac{\ord_p(z)}2\right\rceil\right\}.
\end{equation}
Then the unique primitive representative in the marked similarity
class is
\begin{equation}\label{eq:primitive-representative}
 \left(
 s^2,\left(\frac{su}{v}\right)^2,\frac{s^2w}{z}
 \right).
\end{equation}
\end{theorem}

\begin{proof}
Because the first normalized side is \(1\), a scaling that turns it
into an integer square must multiply all lengths by \(s^2\) for some
\(s\in\Z_{>0}\).  Integrality of the second and third sides requires
\[
 v\mid s,\qquad z\mid s^2.
\]
Fix a prime \(p\), and write
\[
 \beta=\ord_p(v),\qquad
 \delta=\ord_p(z),\qquad c=\ord_p(s).
\]
The least permitted exponent is
\[
 c_0=\max\{\beta,\lceil\delta/2\rceil\}.
\]
At \(c=c_0\), the \(p\)-adic valuations of the three sides in
\eqref{eq:primitive-representative} are
\begin{equation}\label{eq:three-valuations}
 2c,\qquad
 2c+2\ord_p(u)-2\beta,\qquad
 2c+\ord_p(w)-\delta.
\end{equation}
If \(c=0\), the first valuation is zero.  If \(c=\beta>0\), then
\(\ord_p(u)=0\) and the second valuation is zero.  The only remaining
case has \(c>\beta\), hence
\[
 c=\lceil\delta/2\rceil,\qquad \delta>2\beta.
\]
Since then \(\ord_p(w)=0\), the third valuation is zero when
\(\delta\) is even and is one when \(\delta\) is odd.  In the latter
case the other two valuations are positive.  Therefore the minimal
square scaling fails to be primitive at \(p\) precisely when
\begin{equation}\label{eq:bad-prime}
 \delta>2\beta\quad\text{and}\quad\delta\ \text{is odd}.
\end{equation}

On the other hand,
\[
 \ord_p\left(\frac{z}{\gcd(z,v^2)}\right)
 =\max\{\delta-2\beta,0\}.
\]
Thus there is no prime satisfying \eqref{eq:bad-prime} if and only if
the integer in \eqref{eq:primitive-criterion} is a square.  In that
case at least one valuation in \eqref{eq:three-valuations} is zero
for every \(p\), so \eqref{eq:primitive-representative} is primitive.
 If a bad prime exists, all three valuations are already positive at
 the least admissible \(c\), and increasing \(c\) cannot remove the
 common factor.  No square-preserving primitive representative then
 exists.

 Uniqueness follows from the same valuations.  Every admissible square
 scaling \(s'^2\) has
 \(\ord_p(s')\ge c_p\) for each prime \(p\).  If the inequality were
 strict for some \(p\), all three valuations in
 \eqref{eq:three-valuations} would increase by the same positive even
 integer, so the resulting three sides would have a common factor
 \(p\).  Primitivity therefore forces
 \(\ord_p(s')=c_p\) for every \(p\), and hence \(s'=s\).

 The sides in \eqref{eq:primitive-representative} are integral, while
 the area remains rational under square scaling.  Hence
 Lemma~\ref{lem:integer-sides-rational-area} shows that their area is
 an integer.
\end{proof}

\begin{example}\label{ex:nonprimitive}
The following example
\[
 (65^2,52^2,1599)
 \sim
 \left(1,\left(\frac45\right)^2,\frac{123}{325}\right)
\]
has
\[
 \frac{325}{\gcd(325,5^2)}=13,
\]
which is not a square.  Hence its marked class has no primitive
integral representative preserving both square sides.  Dividing the
displayed integral triple by \(13\) gives \((325,208,123)\), but the
first two sides are no longer both squares.
\end{example}

\section{The natural elliptic curve of two-square triangles}
\label{sec:elliptic-model}

Set
\[
 v=r-a.
\]
Using \(a^2+b^2=1\), equation \eqref{eq:normal-form} becomes
\begin{equation}\label{eq:natural-quartic}
 \mathcal Q_b:\quad v^2=q^4-b^2.
\end{equation}
This quartic is independent of the choice of sign of \(b\).

\begin{theorem}[Natural elliptic correspondence]
\label{thm:natural-elliptic-correspondence}
Let \(a,b\) satisfy \eqref{eq:intro-ab}.  The formulas
\begin{equation}\label{eq:quartic-to-elliptic}
 {x=2(q^2+v),\qquad y=4q(q^2+v)}
\end{equation}
and
\begin{equation}\label{eq:elliptic-to-quartic}
 {q=\frac{y}{2x},\qquad
        v=\frac{x}{4}-\frac{b^2}{x}}
\end{equation}
are mutually inverse bijections between the affine rational points of
\(\mathcal Q_b\) and
\[
 E_b(\Q)\setminus\{O,(0,0)\},\qquad
 E_b:\quad y^2=x^3+4b^2x.
\]
The corresponding point gives a nondegenerate triangle if and only if
\begin{equation}\label{eq:admissibility}
 {\rho_{a,b}(x,y)
 =a+\frac{x}{4}-\frac{b^2}{x}>0.}
\end{equation}
The points \((x,y)\) and \((x,-y)\) give the same marked triangle.
\end{theorem}

\begin{proof}
Let \((q,v)\) satisfy \eqref{eq:natural-quartic}, and put
\(t=q^2+v\).  Since
\[
 (q^2+v)(q^2-v)=q^4-v^2=b^2\ne0,
\]
we have \(t\ne0\).  Moreover,
\[
\begin{aligned}
 t^2+b^2
 &=(q^2+v)^2+b^2\\
 &=q^4+2q^2v+v^2+b^2
 =2q^2(q^2+v)=2q^2t.
\end{aligned}
\]
For \(x=2t\) and \(y=4qt\),
\[
 x^3+4b^2x=8t(t^2+b^2)=16q^2t^2=y^2.
\]
Thus \eqref{eq:quartic-to-elliptic} lands on \(E_b\), with \(x\ne0\).

Conversely, let \((x,y)\in E_b(\Q)\) with \(x\ne0\).  The elliptic
equation gives
\[
 \frac{y^2}{4x^2}=\frac{x}{4}+\frac{b^2}{x}.
\]
Defining \(q,v\) by \eqref{eq:elliptic-to-quartic}, we obtain
\[
 q^2+v=\frac{x}{2},\qquad
 q^2-v=\frac{2b^2}{x}.
\]
Their product is \(b^2\), so \(q^4-v^2=b^2\).  These identities also
show directly that the two displayed maps are inverse.

The only affine point of \(E_b\) with \(x=0\) is \((0,0)\), while
\(O\) is the point at infinity.  On the smooth projective completion
of \(\mathcal Q_b\), these are the images of its two points at
infinity.  Finally \(r=a+v\), so \(r>0\) is precisely
\eqref{eq:admissibility}.  Proposition~\ref{prop:marked-normal-form}
then supplies the nondegenerate triangle.  Replacing \(y\) by \(-y\)
replaces \(q\) by \(-q\) and does not change its side lengths.
\end{proof}

\begin{definition}
The \emph{admissible set} on \(E_b\) is
\[
 \Eadm(\Q)=
 \left\{(x,y)\in E_b(\Q):
 x\ne0,\ \rho_{a,b}(x,y)>0\right\}.
\]
\end{definition}

\begin{remark}[What is classified]\label{rem:marked-quotient}
For fixed \(a,b\) and an ordered pair of marked square sides,
Theorem~\ref{thm:natural-elliptic-correspondence} gives a bijection
between marked similarity classes and
\(\Eadm(\Q)/(y\sim-y)\).  Forgetting the marking introduces only
finitely many edge permutations, so the corresponding statement for
unmarked triangles is finite-to-one.  Replacing \(b\) by \(-b\)
reflects the coordinate picture and does not change the triangle.
\end{remark}

For later use, the complete list of exclusions and finite
identifications is:
\[
\begin{array}{@{}L{0.23\textwidth}L{0.37\textwidth}L{0.29\textwidth}@{}}
\toprule
Object & Geometric meaning & Treatment\\
\midrule
\(b=0\), equivalently \(a=\pm1\)
& collinear vertices and a singular cubic & excluded throughout\\
\(O\in E_b\) & one point at infinity of the quartic & no affine triangle\\
\((0,0)\in E_b\) & the other point at infinity & no affine triangle\\
\(r=0\) or \(r<0\) & zero or negative third side
& excluded by \eqref{eq:admissibility}\\
\(y\leftrightarrow-y\) & \(q\leftrightarrow-q\)
& the same marked triangle\\
\(b\leftrightarrow-b\) & reflection in the base line
& the same triangle\\
\bottomrule
\end{array}
\]

\section{The two-square elliptic surface}
\label{sec:sections}

\subsection{Generic model and exact Mordell--Weil group}

Write
\begin{equation}\label{eq:k-parametrization}
 h=k^2+1,\qquad
 a=\frac{2k}{h},\qquad
 b=\frac{k^2-1}{h}.
\end{equation}
Finite \(k\) parametrizes every rational point of \(a^2+b^2=1\)
except \((a,b)=(0,1)\).  This omission causes no unoriented geometric
loss: changing \(b\) to \(-b\) reflects the coordinate picture, leaves
the curve below unchanged, and corresponds to \(k\leftrightarrow1/k\).
Thus \(k=0\), which gives \((a,b)=(0,-1)\), represents the same
triangle geometry as the omitted point.  The values \(k=\pm1\) give
singular, degenerate fibers.  The fiber at \(k=0\) is smooth, but
\[
 Q_0=(2,4)\quad\text{has order \(4\)},\qquad
 P_0=2Q_0=(0,0)\quad\text{has order \(2\)}.
\]

\begin{proposition}\label{prop:elliptic-surface}
The change of variables
\[
 X=h^2x,\qquad Y=h^3y
\]
transforms \(E_b\) into
\begin{equation}\label{eq:elliptic-surface}
 {E_k:\quad Y^2=X^3+4(k^4-1)^2X.}
\end{equation}
Its standard Weierstrass discriminant and \(j\)-invariant are
\begin{equation}\label{eq:surface-invariants}
 \Delta(E_k)=-4096(k^4-1)^6,\qquad j(E_k)=1728.
\end{equation}
The inverse triangle map is
\begin{equation}\label{eq:scaled-inverse}
 {
 q=\frac{Y}{2hX},\qquad
 r=\frac{2k}{h}+\frac{X}{4h^2}
       -\frac{(k^2-1)^2}{X}.}
\end{equation}
The forward map is
\begin{equation}\label{eq:scaled-forward}
 X=2h^2(q^2+r-a),\qquad
 Y=4h^3q(q^2+r-a).
\end{equation}
\end{proposition}

\begin{proof}
Substitution into \(y^2=x^3+4b^2x\) gives
\[
 Y^2=X^3+4h^4b^2X
     =X^3+4(k^4-1)^2X.
\]
For a short Weierstrass equation \(Y^2=X^3+AX\), the standard
discriminant is \(-64A^3\).  Here \(A=4(k^4-1)^2\), which gives
\eqref{eq:surface-invariants}; \(B=0\) gives \(j=1728\).  The
discriminant of the cubic polynomial is one sixteenth of the standard
Weierstrass discriminant.
Equations \eqref{eq:scaled-inverse} and
\eqref{eq:scaled-forward} follow immediately from
\eqref{eq:quartic-to-elliptic}--\eqref{eq:elliptic-to-quartic}.
\end{proof}

\begin{theorem}[Exact generic Mordell--Weil group]
\label{thm:generic-mordell-weil}
Put
\[
 K_0=\Q(k),\qquad K_{\mathrm{geom}}=\overline{\Q}(k),\qquad
 D=2(k^4-1),\qquad
 E_0:\ y^2=x^3+x.
\]
The generic fiber of \eqref{eq:elliptic-surface} is the quadratic
twist
\[
 E_D:\quad Y^2=X^3+D^2X.
\]
Its full Mordell--Weil group over \(K_0\) is
\begin{equation}\label{eq:generic-mw-group}
 {
 E_D(K_0)=\Z Q_k\oplus\langle(0,0)\rangle
 \simeq\Z\oplus\Z/2\Z,}
\end{equation}
where
\[
 Q_k=\bigl(2(k^2+1)(k+1)^2,\,
            4(k^2+1)^2(k+1)^2\bigr).
\]
In particular, the generic rank is exactly \(1\).
\end{theorem}

\begin{proof}
Since \(D=2(k^4-1)\) has a simple zero at \(k=1\), it is not a square
in \(K_0=\Q(k)\).  Thus adjoining \(s\) with \(s^2=D\) gives a genuine
quadratic extension of \(K_0\).  Let \(\mathcal H\) be the smooth
projective model of
\[
 s^2=D=2(k^4-1),
\]
with origin \((k,s)=(1,0)\).  There is an isomorphism over \(\Q\)
\begin{equation}\label{eq:H-to-E0}
 \theta:\mathcal H\longrightarrow E_0,\qquad
 x=\frac{k+1}{k-1},\qquad
 y=\frac{2(k^2+1)(k+1)}{(k-1)s}.
\end{equation}
Indeed, direct substitution gives \(y^2=x^3+x\), and the inverse is
\begin{equation}\label{eq:E0-to-H}
 k=\frac{x+1}{x-1},\qquad
 s=\frac{4y}{(x-1)^2}.
\end{equation}
The formulas extend over the smooth projective models; the chosen
origin maps to \(O\in E_0\).  The deck involution
\(\iota(k,s)=(k,-s)\) corresponds under \(\theta\) to negation on
\(E_0\).

Over \(L_0=K_0(s)=\Q(\mathcal H)\), the twist isomorphism and its inverse
are
\begin{equation}\label{eq:twist-isomorphism}
 E_0\longrightarrow E_D,\quad
 (x,y)\longmapsto(Dx,Dsy),
 \qquad
 (X,Y)\longmapsto\left(\frac XD,\frac{Y}{Ds}\right).
\end{equation}
Consequently, \(E_D(K_0)\) identifies with the group of
\(\Q\)-rational maps \(f:\mathcal H\dashrightarrow E_0\) satisfying
\[
 f\circ\iota=[-1]\circ f.
\]
Every such rational map extends to a morphism because both curves are
smooth and projective.  Put \(T=f(O_{\mathcal H})\).  By the
elliptic-curve rigidity lemma---equivalently, the standard fact that a
morphism of elliptic curves taking the origin to the origin is a group
homomorphism---the map \(\phi=f-T\) belongs to
\(\operatorname{Hom}_{\Q}(\mathcal H,E_0)\)
\cite[Chapter~III, \S4]{Silverman}.  Hence \(f\) has the unique
decomposition
\[
 f=\phi+T,\qquad
 \phi\in\operatorname{Hom}_{\Q}(\mathcal H,E_0),\quad
 T\in E_0(\Q).
\]
The displayed anti-invariance condition is equivalent to \(2T=O\).
Hence
\begin{equation}\label{eq:anti-invariant-maps}
 E_D(K_0)\simeq
 \operatorname{Hom}_{\Q}(\mathcal H,E_0)\oplus E_0(\Q)[2].
\end{equation}
By \eqref{eq:H-to-E0},
\(\operatorname{Hom}_{\Q}(\mathcal H,E_0)
\simeq\operatorname{End}_{\Q}(E_0)\).
The curve \(E_0\) has geometric complex multiplication by
\(\Z[i]\), while complex conjugation fixes only the subring \(\Z\);
thus \cite[Chapter~III, \S9]{Silverman}
\[
 \operatorname{End}_{\Q}(E_0)=\Z.
\]
Moreover,
\[
 E_0(\Q)[2]=\{O,(0,0)\}.
\]
Finally, evaluating the identity map of \(E_0\) at the generic point
of \(\mathcal H\) and applying \eqref{eq:twist-isomorphism} gives
exactly \(Q_k\).  The constant nonzero two-torsion point gives
\((0,0)\).  This proves \eqref{eq:generic-mw-group}.
\end{proof}

\begin{remark}\label{rem:geometric-generic-rank}
The anti-invariant-map description remains valid after extending the
constant field to \(\overline{\Q}\).  Since
\[
 \operatorname{End}_{\overline{\Q}}(E_0)=\Z[i],
\]
one has
\[
 E_D(K_{\mathrm{geom}})
 =E_D(\Q(i)(k))
 \simeq\Z Q_k\oplus\Z R_k\oplus(\Z/2\Z)^2.
\]
Thus the geometric generic rank is exactly \(2\), where
\[
 R_k=\bigl(-2(k^2+1)(k+1)^2,\,
            4i(k^2+1)^2(k+1)^2\bigr),
\]
is the image of the endomorphism \([i](x,y)=(-x,iy)\).
\end{remark}

\begin{proposition}[Kummer--K3 interpretation]
\label{prop:kummer-k3-surface}
Let \(\mathcal S\to\PP^1_k\) be the smooth relatively minimal elliptic
surface whose generic fiber is
\[
 E_D:\qquad
 Y^2=X^3+4(k^4-1)^2X=X^3+D^2X,
 \qquad D=2(k^4-1).
\]
Then:
\begin{enumerate}
\item over \(\overline{\Q}\), the singular fibers of \(\mathcal S\)
are four fibers of type \(I_0^*\), situated at \(k^4=1\), while the
fiber at \(k=\infty\) is smooth;
\item \(\mathcal S\) is a K3 surface,
\[
 \rho(\mathcal S_{\overline{\Q}})=20,\qquad
 E_D(K_{\mathrm{geom}})
 \simeq\Z[i]\oplus(\Z/2\Z)^2,
\]
so its geometric Mordell--Weil rank is \(2\);
\item as a smooth projective surface over \(\Q\),
\begin{equation}\label{eq:kummer-identification}
 {\mathcal S\simeq
 \operatorname{Km}(E_0\times E_0),\qquad E_0:y^2=x^3+x.}
\end{equation}
Under this identification, the displayed elliptic fibration is the
Kummer fibration induced by projection to the second factor followed
by \(E_0\to E_0/\{\pm1\}\).
\end{enumerate}
\end{proposition}

\begin{proof}
For the displayed Weierstrass equation,
\[
 c_4=-192(k^4-1)^2,\qquad
 \Delta=-4096(k^4-1)^6.
\]
At each root \(\zeta\) of \(k^4-1\), the root is simple, the equation
is locally minimal, and
\[
 \ord_{k-\zeta}(c_4)=2,\qquad
 \ord_{k-\zeta}(\Delta)=6.
\]
Tate's algorithm therefore gives a fiber of type \(I_0^*\).  At
infinity put \(t=1/k\) and
\[
 X_\infty=t^4X,\qquad Y_\infty=t^6Y.
\]
The equation becomes
\[
 Y_\infty^2=X_\infty^3+4(1-t^4)^2X_\infty,
\]
whose discriminant is nonzero at \(t=0\).  Thus the fiber at infinity
is smooth and there are no further singular fibers.

Since \(e(I_0^*)=6\),
\[
 e(\mathcal S)=4\cdot6=24,\qquad
 \chi(\mathcal O_{\mathcal S})=2.
\]
The fibration has a section and hence no multiple fibers.  The
canonical bundle formula for an elliptic surface over \(\PP^1\) gives
\[
 K_{\mathcal S}\simeq
 \pi^*\mathcal O_{\PP^1}
 \bigl(\chi(\mathcal O_{\mathcal S})-2\bigr)
 \simeq\mathcal O_{\mathcal S},
\]
and \(q(\mathcal S)=0\); hence \(\mathcal S\) is a K3 surface
\cite{SchuettShioda}.

The anti-invariant-map calculation in
\eqref{eq:anti-invariant-maps}, now over \(\overline{\Q}\), gives
\[
 E_D(K_{\mathrm{geom}})
 \simeq
 \operatorname{End}_{\overline{\Q}}(E_0)
 \oplus E_0(\overline{\Q})[2]
 \simeq\Z[i]\oplus(\Z/2\Z)^2.
\]
Every \(I_0^*\)-fiber contributes a \(D_4\) root lattice.  The
Shioda--Tate formula therefore yields
\[
 \rho(\mathcal S_{\overline{\Q}})
 =2+4\cdot4+2=20.
\]

It remains to identify the surface.  On \(E_0\times\mathcal H\), with
\[
 E_0:\ y_1^2=x_1^3+x_1,\qquad
 \mathcal H:\ s^2=D,
\]
consider
\[
 \alpha(P,(k,s))=(-P,(k,-s)).
\]
The invariant functions \(k,x_1,w=sy_1\) satisfy
\[
 w^2=D(x_1^3+x_1).
\]
The change \(X=Dx_1,\ Y=Dw\) gives
\[
 Y^2=X^3+D^2X.
\]
Thus \((E_0\times\mathcal H)/\langle\alpha\rangle\) is birational over
\(\Q\) to the Weierstrass model of \(\mathcal S\).  Under the
isomorphism \(\theta:\mathcal H\to E_0\) in
\eqref{eq:H-to-E0}, the involution \(\alpha\) becomes
\[
 (P,Q)\longmapsto(-P,-Q)
\]
on \(E_0\times E_0\).  Its fixed locus is
\(E_0[2]\times E_0[2]\), consisting of sixteen geometric points.
Consequently the quotient has the sixteen ordinary double points of
the Kummer quotient, and its minimal resolution is
\(\operatorname{Km}(E_0\times E_0)\).  The smooth minimal resolutions
on both sides are K3 surfaces, so the birational identification extends
to the isomorphism \eqref{eq:kummer-identification}.  Projection to
\(\mathcal H/\langle\iota\rangle\simeq\PP^1_k\) gives the asserted
elliptic fibration.
\end{proof}

\subsection{Sections, pointwise non-torsion, and fixed-fiber infinitude}

\begin{proposition}[The isosceles section and its double]
\label{prop:isosceles-section}
On \(E_b\), set
\[
 \widetilde Q=(2(1+a),4(1+a)).
\]
On \(E_k\), the corresponding point is
\[
 Q_k=\bigl(2(k^2+1)(k+1)^2,\,
           4(k^2+1)^2(k+1)^2\bigr).
\]
If
\[
 P_k=(4k^2,4k(k^4+1)),
\]
then
\begin{equation}\label{eq:double-section}
 {2Q_k=P_k.}
\end{equation}
For \(k>0\), \(k\ne1\), the point \(Q_k\) corresponds to the
isosceles triangle
\[
 \left(1,1,\frac{4k}{k^2+1}\right).
\]
\end{proposition}

\begin{proof}
The relation \(a^2+b^2=1\) gives
\[
 16(1+a)^2=8(1+a)^3+8b^2(1+a),
\]
so \(\widetilde Q\in E_b(\Q)\).  Its tangent slope is
\[
 \lambda
 =\frac{3[2(1+a)]^2+4b^2}{2\cdot4(1+a)}
 =a+2.
\]
The doubling formulas give
\[
 x(2\widetilde Q)=\lambda^2-4(1+a)=a^2
\]
and
\[
 y(2\widetilde Q)
 =\lambda\bigl(2(1+a)-a^2\bigr)-4(1+a)
 =a(2-a^2).
\]
Multiplication by \(h^2,h^3\) yields
\[
 h^2a^2=4k^2,\qquad
 h^3a(2-a^2)=4k(k^4+1),
\]
which proves \eqref{eq:double-section}.  Substitution of \(Q_k\) in
\eqref{eq:scaled-inverse} gives \(q=1\) and
\(r=4k/(k^2+1)\).  For \(k>0\), \(k\ne1\), one has \(0<r<2\), so
the triangle is nondegenerate.
\end{proof}

\begin{theorem}[Uniform non-torsion]\label{thm:uniform-nontorsion}
For every
\[
 k\in\Q\setminus\{0,\pm1\},
\]
the point \(P_k\) has infinite order on \(E_k(\Q)\).  Hence \(Q_k\)
also has infinite order and
\[
 \rank E_k(\Q)\ge1.
\]
\end{theorem}

\begin{proof}
Write \(k=m/n\), where
\[
 m,n\in\Z,\qquad \gcd(m,n)=1,\qquad
 mn(m-n)(m+n)\ne0.
\]
The substitution \(X=U/n^4\), \(Y=V/n^6\) gives the integral model
\begin{equation}\label{eq:integral-fiber}
 \mathcal E_{m,n}^{\mathrm{int}}:\quad
 V^2=U^3+4(m^4-n^4)^2U
\end{equation}
and the integral point
\[
 P'=\bigl(4m^2n^2,\,4mn(m^4+n^4)\bigr).
\]
Its \(V\)-coordinate is nonzero because
\(mn(m^4+n^4)\ne0\); hence \(P'\) is not a two-torsion point.
Put
\[
 D=m^4-n^4,\qquad S=m^4+n^4.
\]
The Nagell--Lutz theorem in its cubic-discriminant form
\cite[Chapter~VII, \S3]{Silverman} says that if a non-two-torsion
integral point on \(V^2=U^3+AU+B\) is torsion, then its
\(V\)-coordinate squared divides \(4A^3+27B^2\).  If \(P'\) were
torsion, we would therefore have
\[
 16m^2n^2S^2\mid256D^6,
\]
or
\begin{equation}\label{eq:nagell-divisibility}
 m^2n^2S^2\mid16D^6.
\end{equation}

Since \(\gcd(m,n)=1\),
\[
 \gcd(m,D)=\gcd(n,D)=1.
\]
Equation \eqref{eq:nagell-divisibility} forces
\(m^2\mid16\) and \(n^2\mid16\).  Thus
\(\abs m,\abs n\in\{1,2,4\}\).  Coprimality and
\(\abs m\ne\abs n\) leave only cases in which one absolute value is
\(1\) and the other is \(2\) or \(4\).  In all of them \(S>1\) is
odd, while
\[
 \gcd(S,D)\mid2,\qquad \gcd(S,16)=1.
\]
Hence \(S^2\) is coprime to the right side of
 \eqref{eq:nagell-divisibility}, a contradiction.  Thus \(P'\), and
 therefore \(P_k\), has infinite order.  The identity \(2Q_k=P_k\)
 gives the same conclusion for \(Q_k\).
\end{proof}

\begin{theorem}[Fixed-fiber infinitude]
\label{thm:fixed-fiber-infinitude}
For every \(k\in\Q\setminus\{0,\pm1\}\), there are infinitely many
pairwise nonsimilar scalene rational Heron triangles with exactly two
square sides on the fixed \(k\)-fiber.  Each can be multiplied by a
rational square so that its sides are integers, exactly two of them
are integer squares, and its area is an integer.
\end{theorem}

\begin{proof}
Put \(d=k^4-1\ne0\).  The cubic
\[
 X^3+4d^2X=X(X^2+4d^2)
\]
has exactly one real root.  Hence \(E_k(\mathbb R)\) is connected and,
as a compact real Lie group, is a circle.  By
Theorem~\ref{thm:uniform-nontorsion}, \(P_k\) has infinite order.
The closure of \(\langle P_k\rangle\) is an infinite closed subgroup
of a circle and is therefore the whole circle.

Consider the real open set
\[
\mathcal U_k=
\left\{(X,Y)\in E_k(\mathbb R):
 X\ne0,\ 
 \frac{2k}{k^2+1}
 +\frac{X}{4(k^2+1)^2}
 -\frac{(k^2-1)^2}{X}>0
\right\}.
\]
It is nonempty: for all sufficiently large positive \(X\), the
displayed expression is positive and the curve has two real points
above \(X\).  Density gives infinitely many rational multiples of
\(P_k\) in \(\mathcal U_k\).

Only finitely many of these points give a non-scalene triangle.  In
the quartic model:
\[
\begin{array}{ll}
q^2=1 &\Longrightarrow\ v^2=a^2,\\
r=1 &\Longrightarrow\ q^4=2-2a,\\
r=q^2 &\Longrightarrow\ 2aq^2=1,
\end{array}
\]
with the last equation having no solution when \(a=0\).  Each row
therefore contributes only finitely many points.  Removing them
leaves infinitely many scalene marked classes.  The sign quotient and
the finite set of edge permutations cannot turn an infinite set into
a finite one.

For this fixed \(k\), the corresponding \(a,b\) are fixed.  By
Corollary~\ref{cor:fixed-fiber-finiteness}, only finitely many of the
remaining classes have a third square side.  Removing this finite set
leaves infinitely many scalene classes with exactly two square sides.

Finally, write \(q=u/v\) and \(r=w/z\) in lowest terms.  Choose a
positive integer \(s\) divisible by \(vz\).  Scaling by \(s^2\) gives
\[
 s^2,\qquad \left(\frac{su}{v}\right)^2,\qquad
 \frac{s^2w}{z},
\]
which are integral and retain the first two square sides.  The third
side remains nonsquare: if \(s^2r\) were a rational square, then so
would \(r\).  The area is rational and hence integral by
Lemma~\ref{lem:integer-sides-rational-area}.  No primitivity assertion
is made here; that issue is governed by
Theorem~\ref{thm:primitive-scaling}.
\end{proof}

\subsection{The equal-square isosceles branch}

\begin{theorem}[Primitive isosceles classification]
\label{thm:primitive-isosceles-classification}
Every nondegenerate rational Heron isosceles triangle whose equal
sides are squares is similar to
\begin{equation}\label{eq:isosceles-normal-form}
 \left(1,1,\frac{4k}{k^2+1}\right),
 \qquad k\in\Q_{>0},\quad k\ne1.
\end{equation}
The parameters \(k\) and \(1/k\) give the same class; every such
class is represented uniquely if one imposes \(k>1\).
Write \(k=m/n\) with coprime positive integers \(m,n\), and put
\[
 g=\gcd(4,m^2+n^2)=
 \begin{cases}
 1,&m,n\text{ of opposite parity},\\
 2,&m,n\text{ both odd}.
 \end{cases}
\]
The class has a primitive integral representative with equal square
sides if and only if
\begin{equation}\label{eq:isosceles-square-condition}
 {\frac{m^2+n^2}{g}=d^2}
\end{equation}
for some \(d\in\Z_{>0}\).  In that case its unique primitive
representative, up to edge order, is
\begin{equation}\label{eq:isosceles-primitive}
 {\left(d^2,d^2,\frac{4mn}{g}\right),}
\end{equation}
with area
\begin{equation}\label{eq:isosceles-area}
 {\mathcal A=\frac{2mn\abs{m^2-n^2}}{g^2}.}
\end{equation}
\end{theorem}

\begin{proof}
After normalizing the equal sides to \(1\), rationality of the area
is equivalent to rationality of the altitude.  If the base is \(r\),
then
\[
 \left(\frac r2\right)^2+t^2=1.
\]
The rational parametrization of the unit circle gives
\[
 \frac r2=\frac{2k}{k^2+1},
\]
which proves \eqref{eq:isosceles-normal-form}; \(k=1\) is the
degenerate value \(r=2\).

For \(k=m/n\), the base in lowest terms is
\[
 \frac{4mn/g}{(m^2+n^2)/g}.
\]
Apply Theorem~\ref{thm:primitive-scaling} with \(q=1\), so the
denominator \(v\) of the second square root is \(1\).  A
square-preserving primitive representative exists exactly when the
reduced base denominator is a square, which is
\eqref{eq:isosceles-square-condition}.  The minimal square scaling is
\(d^2\), giving \eqref{eq:isosceles-primitive}.

Half the base is \(2mn/g\), and the altitude is
\[
 \sqrt{d^4-\frac{4m^2n^2}{g^2}}
 =\frac{\abs{m^2-n^2}}{g}.
\]
Their product is \eqref{eq:isosceles-area}.
\end{proof}

\begin{corollary}[Two primitive isosceles families]
\label{cor:two-isosceles-families}
Let \(u>v>0\) be coprime and of opposite parity.  All primitive
triangles in Theorem~\ref{thm:primitive-isosceles-classification}
occur, allowing the usual interchange of the two legs in a primitive
Pythagorean parametrization, in the two families
\begin{equation}\label{eq:isosceles-family-one}
 {\bigl((u^2+v^2)^2,(u^2+v^2)^2,
              8uv(u^2-v^2)\bigr)}
\end{equation}
and
\begin{equation}\label{eq:isosceles-family-two}
 {\bigl((u^2+v^2)^2,(u^2+v^2)^2,
              2\abs{u^4-6u^2v^2+v^4}\bigr).}
\end{equation}
Their areas are
\begin{equation}\label{eq:isosceles-family-area}
 4uv\abs{u^2-v^2}\,
 \abs{u^4-6u^2v^2+v^4}.
\end{equation}
\end{corollary}

\begin{proof}
If \(g=1\), condition \eqref{eq:isosceles-square-condition} is the
primitive Pythagorean equation
\[
 m^2+n^2=d^2.
\]
Thus, up to interchanging \(m,n\),
\[
 m=u^2-v^2,\qquad n=2uv,\qquad d=u^2+v^2,
\]
and \eqref{eq:isosceles-primitive} gives
\eqref{eq:isosceles-family-one}.

If \(g=2\), then \(m,n\) are odd and
\[
 m^2+n^2=2d^2.
\]
Put
\[
 A=\frac{m+n}{2},\qquad B=\frac{\abs{m-n}}2.
\]
The integers \(A,B\) are coprime and of opposite parity, and
\(A^2+B^2=d^2\).  Hence, allowing their interchange,
\[
 A=u^2-v^2,\qquad B=2uv,\qquad d=u^2+v^2.
\]
 Because
\[
 \frac{4mn}{g}=2\abs{A^2-B^2}
 =2\abs{u^4-6u^2v^2+v^4},
\]
we get \eqref{eq:isosceles-family-two}.  Formula
 \eqref{eq:isosceles-family-area} follows from
 \eqref{eq:isosceles-area}.

 Conversely, every coprime opposite-parity pair \(u>v>0\) produces a
 nondegenerate primitive triangle in each displayed family.  For
 \eqref{eq:isosceles-family-one}, take
 \((m,n)=(u^2-v^2,2uv)\), allowing their interchange.  For
 \eqref{eq:isosceles-family-two}, take
 \[
  (m,n)=\bigl(u^2-v^2+2uv,\,
              \abs{u^2-v^2-2uv}\bigr).
 \]
 These pairs satisfy the corresponding coprimality and parity
 hypotheses; none of the displayed positive quantities vanishes,
 since that would give an irrational value for \(u/v\).
 Theorem~\ref{thm:primitive-isosceles-classification} then proves
 nondegeneracy and primitivity.  For example, \((u,v)=(4,1)\) in the
 second family gives \((17^2,17^2,322)\).
\end{proof}

\Needspace{26\baselineskip}
\section{An explicit primitive scalene family and its counting law}
\label{sec:primitive-family}

\begin{theorem}[A primitive family with exactly two square sides]
\label{thm:low-degree-primitive-family}
Let \(m>n>0\) be coprime and define
\begin{equation}\label{eq:LMR}
\begin{aligned}
 L&=2mn(m^2+n^2),\\
 M&=m^4+n^4,\\
 R&=-m^8-n^8+8m^5n^3+6m^4n^4+8m^3n^5.
\end{aligned}
\end{equation}
Assume
\begin{equation}\label{eq:positive-range}
 {\frac mn+\frac nm<1+\sqrt3.}
\end{equation}
Set
\begin{equation}\label{eq:epsilon}
 \varepsilon=
 \begin{cases}
 1,&m,n\text{ of opposite parity},\\
 4,&m,n\text{ both odd}.
 \end{cases}
\end{equation}
Then
\begin{equation}\label{eq:low-family}
 {\left(
 \frac{L^2}{\varepsilon},
 \frac{M^2}{\varepsilon},
 \frac{R}{\varepsilon}\right)}
\end{equation}
is a primitive scalene Heron triangle with exactly two square sides.
Its area is
\begin{equation}\label{eq:low-family-area}
 {\mathcal A_{m,n}
 =\frac{
 2m^2n^2(m^2+n^2)\abs{m^2-n^2}R
 }{\varepsilon^2}.}
\end{equation}
\end{theorem}

\begin{proof}
Take \(k=m/n\) and the point
\[
 P_k=(4k^2,4k(k^4+1))
\]
of Proposition~\ref{prop:isosceles-section}.  Substitution in
\eqref{eq:scaled-inverse} gives
\[
 q=\frac{m^4+n^4}{2mn(m^2+n^2)}=\frac ML
\]
and
\[
 r=\frac{R}{4m^2n^2(m^2+n^2)^2}=\frac{R}{L^2}.
\]
Thus the normalized sides are
\[
 1,\qquad \left(\frac ML\right)^2,\qquad \frac{R}{L^2},
\]
and multiplication by \(L^2\) gives \((L^2,M^2,R)\).

For positivity, put
\[
 T=\frac mn+\frac nm=\frac{m^2+n^2}{mn}.
\]
Direct factorization gives
\begin{equation}\label{eq:R-factorization}
 {R=-m^4n^4(T^2+2T+2)(T^2-2T-2).}
\end{equation}
Since \(T\ge2\), the first quadratic factor is positive and
\[
 R>0\quad\Longleftrightarrow\quad T<1+\sqrt3.
\]
This is \eqref{eq:positive-range}.  Also
\[
 b=\frac{m^2-n^2}{m^2+n^2}\ne0.
\]
The coordinate construction in
Proposition~\ref{prop:marked-normal-form} therefore proves all strict
triangle inequalities.

The normalized area is
\[
 \frac{\abs b r}{2}
 =\frac{\abs{m^2-n^2}R}
 {2(m^2+n^2)L^2}.
\]
Scaling lengths by \(L^2\) scales area by \(L^4\).  After the
integrality verification below, division of all three sides by
\(\varepsilon\) divides the area by \(\varepsilon^2\), yielding
\eqref{eq:low-family-area}.

We next prove primitivity.  One has
\begin{equation}\label{eq:gcd-LM}
 \gcd(L,M)=
 \begin{cases}
 1,&m,n\text{ of opposite parity},\\
 2,&m,n\text{ both odd}.
 \end{cases}
\end{equation}
Indeed, if an odd prime \(p\) divided both \(L\) and \(M\), then
\(p\mid m\) or \(p\mid n\) would immediately contradict
\(M=m^4+n^4\) and \(\gcd(m,n)=1\).  If instead
\(p\mid m^2+n^2\), then
\[
 M\equiv2n^4\not\equiv0\pmod p,
\]
again a contradiction.  The power of \(2\) in
\eqref{eq:gcd-LM} follows directly from parity.

When \(m,n\) have opposite parity, the first two sides \(L^2,M^2\)
are already coprime.  When both are odd,
\[
 \gcd(L,M)=2,\qquad
 \ord_2(L)=2,\qquad \ord_2(M)=1,
\]
and the congruence computation below gives
\[
 R\equiv20\pmod{32}.
\]
Thus
\[
 \ord_2(L^2)=4,\qquad
 \ord_2(M^2)=\ord_2(R)=2.
\]
There is no common odd prime because \(\gcd(L,M)=2\), so the three
sides have greatest common divisor exactly \(4\).  After division by
\(4\), the first two sides are
\((L/2)^2,(M/2)^2\), with coprime square roots.  This proves
primitivity in both cases.
The displayed area is rational, and the resulting sides are
integers; hence Lemma~\ref{lem:integer-sides-rational-area} also shows
that the area in \eqref{eq:low-family-area} is an integer.

It remains to show that the third side is not a square.  If \(m,n\)
have opposite parity, every term in \(R\) except one negative eighth
power vanishes modulo \(8\), so
\[
 R\equiv7\pmod8.
\]
If \(m,n\) are both odd, then modulo \(32\),
\[
 m^8\equiv n^8\equiv1,\qquad
 6m^4n^4\equiv6,
\]
while
\[
 8m^5n^3+8m^3n^5
 =8m^3n^3(m^2+n^2)\equiv16.
\]
Consequently,
\[
 R\equiv20\pmod{32},\qquad \frac R4\equiv5\pmod8.
\]
Neither possible third side is a square.

Finally, the first two square sides are unequal.  If \(L=M\), then
\(k=m/n\) is a rational root of
\[
 k^4-2k^3-2k+1=0.
\]
The rational-root theorem leaves only \(\pm1\), neither of which is a
root.  Since the third side is not a square, it is unequal to either
of the first two sides.  The triangle is therefore scalene and has
exactly two square sides.
\end{proof}

\begin{example}\label{ex:small-low-family}
The first two convenient specializations are
\[
\begin{array}{ccl}
(m,n)=(2,1)
&\Longrightarrow&(20^2,17^2,159),
\qquad \mathcal A=19080,\\[2mm]
(m,n)=(3,2)
&\Longrightarrow&(156^2,97^2,23423),
\qquad \mathcal A=109619640.
\end{array}
\]
\end{example}

\begin{theorem}[Counting lower bound]\label{thm:counting-lower-bound}
Let \(N(X)\) be the number of unmarked primitive scalene Heron
triangles with exactly two square sides and largest side at most
\(X\).  Then, as \(X\to\infty\),
\begin{equation}\label{eq:counting-asymptotic}
 {
 N(X)\ge
 \left(\frac{3}{\pi^2\,416^{1/4}}+o(1)\right)X^{1/4}.}
\end{equation}
In particular, \(N(X)\gg X^{1/4}\).
\end{theorem}

\begin{proof}
Restrict Theorem~\ref{thm:low-degree-primitive-family} to
\[
 1\le n\le B,\qquad n<m<2n,\qquad \gcd(m,n)=1.
\]
Then
\[
 \frac mn+\frac nm<\frac52<1+\sqrt3,
\]
so every pair is admissible.  Since \(m<2n\le2B\),
\[
 L<20B^4,\qquad M<17B^4,
\]
and
\[
\begin{aligned}
 R
 &<8m^5n^3+6m^4n^4+8m^3n^5\\
 &<(256+96+64)B^8=416B^8.
\end{aligned}
\]
Thus every resulting triangle has largest side at most \(416B^8\).

It remains to prove that the parameter pairs give different unmarked
triangles.  For \(k=m/n\in(1,2)\), the ratio of the smaller to the
larger square root of the two square sides is
\[
 \frac{\sqrt{M^2/\varepsilon}}
      {\sqrt{L^2/\varepsilon}}
 =\frac ML
 =\eta(k)=\frac{k^4+1}{2k(k^2+1)}.
\]
Here \(\eta(k)<\eta(2)=17/20<1\), and
\[
 \eta'(k)=
 \frac{(k^2-1)(k^4+4k^2+1)}
 {2k^2(k^2+1)^2}>0.
\]
Because the third side is not a square, \(\eta(k)\) is intrinsic to
the unmarked triangle.  Hence distinct reduced fractions \(k\) give
distinct triangles.

For each \(n\), writing \(m=n+j\) turns the parameter conditions into
\[
 1\le j<n,\qquad \gcd(j,n)=1.
\]
The number of pairs is therefore
\[
 \sum_{n\le B}\varphi(n)+O(1)
 =\frac{3}{\pi^2}B^2+O(B\log B);
\]
see, for example, \cite[Chapter~3]{Apostol}.  Taking
\[
 B=\left\lfloor\left(\frac{X}{416}\right)^{1/8}\right\rfloor
\]
gives
\eqref{eq:counting-asymptotic}.
\end{proof}

\section{The all-square genus-three cover}
\label{sec:genus-three}

We now impose the additional condition that the third normalized side
be a square.  Write \(r=z^2\), with \(z\in\Q^\times\).  Equation
\eqref{eq:normal-form} becomes
\begin{equation}\label{eq:genus-three-affine}
 {C_{a,b}:\quad q^4=z^4-2az^2+1,}
\end{equation}
where \(a^2+b^2=1\) and \(b\ne0\).

\begin{proposition}[Primitive representative of an all-square class]
\label{prop:all-square-primitive-representative}
Every rational similarity class of nondegenerate all-square Heron
triangles contains a unique primitive integral all-square
representative, up to edge order.
\end{proposition}

\begin{proof}
Write the three rational square sides as \(x^2,y^2,z^2\), with
\(x,y,z\in\Q_{>0}\).  Multiplication of all lengths by a common
integer square clearing the denominators of \(x,y,z\) produces
\[
 (A^2,B^2,C^2),\qquad A,B,C\in\Z_{>0}.
\]
Let \(g=\gcd(A,B,C)\).  Dividing all three lengths by \(g^2\) leaves
integer square sides and makes their greatest common divisor
\[
 \gcd\left((A/g)^2,(B/g)^2,(C/g)^2\right)=1.
\]
The area remains rational under both similarity scalings, and is
therefore an integer by
Lemma~\ref{lem:integer-sides-rational-area}.  This proves existence.

For uniqueness, suppose two primitive integral triples in the same
ordered similarity class differ by a positive rational factor
\(\lambda=r/s\) in lowest terms.  Integrality forces \(s\) to divide
all sides of the first triple, hence \(s=1\); applying the inverse
scaling forces \(r=1\).  Thus \(\lambda=1\).  Forgetting the edge
order gives precisely the stated qualification.
\end{proof}

\begin{theorem}[All-square criterion]\label{thm:all-square-criterion}
Fix an ordered choice of the first two square sides and the oriented
angle parameter \((a,b)\).  The resulting marked rational similarity
classes of Heron triangles with three square sides are equivalent to
rational points \((z,q)\in C_{a,b}(\Q)\) with \(z\ne0\), modulo the
sign changes \(z\mapsto-z\) and \(q\mapsto-q\).  The normalized
triangle and area are
\[
 (1,q^2,z^2),\qquad \frac{\abs b z^2}{2}.
\]
Forgetting the marking is finite-to-one; the other ordered edge
markings generally belong to curves with different angle parameters.
\end{theorem}

\begin{proof}
This is Proposition~\ref{prop:marked-normal-form} with \(r=z^2\).
Conversely, a point of \eqref{eq:genus-three-affine} gives the vertices
\[
 (0,0),\qquad(z^2,0),\qquad(a,b),
\]
whose sides are \(z^2,1,q^2\) and whose area is
\(\abs b z^2/2>0\).  Also \(q\ne0\), because
\[
 q^4=(z^2-a)^2+b^2>0.
\]
The sign changes do not alter the side lengths.  Permuting the edge
marking is a finite operation, but it generally changes \((a,b)\).
\end{proof}

\begin{proposition}[A smooth plane quartic]
\label{prop:all-square-genus-three}
The projective curve
\begin{equation}\label{eq:genus-three-projective}
 Q^4=Z^4-2aZ^2W^2+W^4
 \subset\PP^2
\end{equation}
is smooth.  It is therefore a nonhyperelliptic curve of genus \(3\).
\end{proposition}

\begin{proof}
Let
\[
 F(Z,Q,W)=Q^4-Z^4+2aZ^2W^2-W^4.
\]
Its partial derivatives are
\[
 F_Q=4Q^3,\qquad
 F_Z=4Z(aW^2-Z^2),\qquad
 F_W=4W(aZ^2-W^2).
\]
At a singular point, \(Q=0\).  Neither \(Z\) nor \(W\) can vanish,
because the equation would then force both to vanish.  The remaining
two derivative equations give
\[
 Z^2=aW^2,\qquad W^2=aZ^2,
\]
and hence \(a^2=1\).  This contradicts \(1-a^2=b^2\ne0\).
Thus the plane quartic is smooth.  A smooth plane quartic has genus
\((4-1)(4-2)/2=3\), and its plane embedding is canonical, so it is
nonhyperelliptic.
\end{proof}

\begin{corollary}[Fiberwise finiteness]\label{cor:fixed-fiber-finiteness}
For each fixed \(a,b\) satisfying \eqref{eq:intro-ab},
\[
 \#C_{a,b}(\Q)<\infty.
\]
\end{corollary}

\begin{proof}
This is Faltings's theorem \cite{Faltings}, applied to the smooth
genus-three curve of
Proposition~\ref{prop:all-square-genus-three}.
\end{proof}

\begin{remark}
Corollary~\ref{cor:fixed-fiber-finiteness} is only a fiberwise
statement.  It gives no finiteness result for the union of
\(C_{a,b}(\Q)\) as \(a,b\) vary.
\end{remark}

\begin{proposition}[Trivial points and exact marked class count]
\label{prop:exact-marked-class-count}
In the projective coordinates of
\eqref{eq:genus-three-projective}, put
\[
 \mathcal T=
 \{[0:1:1],[0:-1:1],[1:1:0],[1:-1:0]\}.
\]
These are exactly the rational points with \(ZW=0\), and
\(C_{a,b}\) has no rational point with \(Q=0\).  The Klein four-group
generated by the sign changes of \(Z\) and \(Q\) acts freely on
\(C_{a,b}(\Q)\setminus\mathcal T\).  Consequently, for fixed oriented
\((a,b)\) and a fixed ordered marking,
\begin{equation}\label{eq:exact-marked-count}
 {
 N^{\mathrm{marked}}_{a,b}
 =\frac{\#C_{a,b}(\Q)-4}{4}.}
\end{equation}
\end{proposition}

\begin{proof}
When \(Z=0\), the equation gives \(Q^4=W^4\), yielding the first two
points of \(\mathcal T\); when \(W=0\), it gives \(Q^4=Z^4\), yielding
the other two.  If \(Q=0\) and \(W\ne0\), then in affine coordinates
\[
 0=z^4-2az^2+1=(z^2-a)^2+b^2,
\]
which is impossible over \(\Q\) because \(b\ne0\); the case \(W=0\)
is also impossible.

Every remaining rational point is affine with \(zq\ne0\).  The three
nonidentity sign changes can fix a point only when, respectively,
\(z=0\), \(q=0\), or \(W=0\).  Hence every remaining orbit has four
points.  By Theorem~\ref{thm:all-square-criterion}, each such orbit is
exactly one marked class, proving \eqref{eq:exact-marked-count}.
\end{proof}

\subsection{The first double cover}

For clarity, write
\[
 \mathcal Q_b:\quad v^2=x^4-b^2
\]
for the genus-one quartic of Section~\ref{sec:elliptic-model}.

\begin{proposition}\label{prop:double-cover}
The map
\begin{equation}\label{eq:first-double-cover}
 \pi_1:C_{a,b}\longrightarrow\mathcal Q_b,\qquad
 (z,q)\longmapsto(x,v)=(q,z^2-a)
\end{equation}
extends to a degree-two morphism of smooth projective curves.  It is
the quotient by
\[
 \iota_z(z,q)=(-z,q).
\]
Its four geometric branch points are
\[
 (z,q)=(0,\zeta),\qquad \zeta^4=1,
\]
and it is unramified at infinity.
\end{proposition}

\begin{proof}
On \(C_{a,b}\),
\[
 (z^2-a)^2=z^4-2az^2+a^2=q^4-b^2,
\]
so \eqref{eq:first-double-cover} lands on \(\mathcal Q_b\).  The
function fields satisfy
\[
 \Q(C_{a,b})=\Q(\mathcal Q_b)(z),\qquad z^2=v+a.
\]
This is generically a quadratic extension, with nontrivial
automorphism \(z\mapsto-z\).

The affine fixed points have \(z=0\), hence \(q^4=1\), and they are
simple ramification points.  At infinity the points have
\([Z:Q:0]\) with \(ZQ\ne0\) and \(Q^4=Z^4\).  Equality
\([-Z:Q:0]=[Z:Q:0]\) would require the projective scaling factor to
be simultaneously \(-1\) from the first coordinate and \(1\) from
the second, which is impossible.  Thus no point at infinity is fixed.
The rational map extends uniquely between the smooth projective
models.
\end{proof}

Riemann--Hurwitz independently gives genus one for the quotient:
\[
 2\cdot3-2=2(2g-2)+4\quad\Longrightarrow\quad g=1.
\]
For affine rational points, the complete lifting condition is
\begin{equation}\label{eq:square-lift}
 {\pi_1^{-1}(x,v)(\Q)\ne\varnothing
 \quad\Longleftrightarrow\quad v+a\in\Q^2,}
\end{equation}
where \(\Q^2=\{t^2:t\in\Q\}\) includes \(0\).  Such a lift produces a
nondegenerate triangle if and only if
\[
 v+a\in\Q^{\times2};
\]
the omitted case \(v=-a\) has \(z=0\) and is one of the trivial
ramification points.

\subsection{Ciani symmetry and the three elliptic quotients}

The quartic \eqref{eq:genus-three-projective} is the Ciani subfamily
studied by Gu\`ardia: after putting \(n=(a+1)/2\) and identifying
\((Y,X,Z)=(Q,Z,W)\), it becomes
\[
 C_n:\quad Y^4=X^4-(4n-2)X^2Z^2+Z^4.
\]
The geometric elliptic quotients and the associated Jacobian splitting
were constructed in \cite{GuardiaJLMS,GuardiaJTNB}.  We retain
explicit \(\Q\)-models and quotient maps because their compatibility
with the Heron parametrization is essential for the arithmetic
applications below.

Let
\[
 \iota_q(z,q)=(z,-q),\qquad
 G=\langle\iota_z,\iota_q\rangle
 \simeq(\Z/2\Z)^2,
\]
and define
\begin{equation}\label{eq:F-a}
 F_a:\quad V^2=2(U-a)(U-1)(U+1).
\end{equation}
Since \(a\ne\pm1\), this is an elliptic curve, with its point at
infinity as origin.  The change \(x=2U,\ y=2V\) gives the split cubic
Weierstrass model
\begin{equation}\label{eq:F-a-standard}
 F'_a:\quad y^2=(x-2a)(x-2)(x+2).
\end{equation}
Its standard invariants are
\begin{equation}\label{eq:F-a-invariants}
 {\Delta(F'_a)=2^{12}b^4,\qquad
 j(F'_a)=64\,\frac{(a^2+3)^3}{b^4}.}
\end{equation}
In particular, \(F'_a\) is nonsingular and has full rational
two-torsion.

\begin{theorem}[Three elliptic quotients]
\label{thm:three-elliptic-quotients}
The quotients by the three nontrivial order-two subgroups of \(G\)
are given by the following degree-two maps:
\begin{align}
 \phi_1(z,q)
 &=\left(q,z^2-a\right)
 &&\text{(quotient by \(\iota_z\))},\label{eq:quotient-one}\\
 \phi_2(z,q)
 &=\left(\frac qz,\frac1{z^2}-a\right)
 &&\text{(quotient by \(\iota_z\iota_q\))},\label{eq:quotient-two}\\
 \phi_3(z,q)
 &=(U,V),\quad
 U=\frac{q^2+1}{z^2},\quad
 V=z(U^2-1)
 &&\text{(quotient by \(\iota_q\))}.\label{eq:quotient-three}
\end{align}
The first two targets are \(\mathcal Q_b\), and the third is \(F_a\).
\end{theorem}

\begin{proof}
The first map is Proposition~\ref{prop:double-cover}.  For the second,
put
\[
 x=\frac qz,\qquad v=\frac1{z^2}-a.
\]
Using \eqref{eq:genus-three-affine},
\[
\begin{aligned}
 x^4-b^2
 &=\frac{q^4-b^2z^4}{z^4}\\
 &=\frac{a^2z^4-2az^2+1}{z^4}
 =\left(\frac1{z^2}-a\right)^2=v^2.
\end{aligned}
\]
Both functions are invariant under \((z,q)\mapsto(-z,-q)\), and
generically
\[
 z^2=\frac1{v+a},\qquad q=xz,
\]
so the fibers have degree two.

For the third map, first pass to the quotient coordinate \(w=q^2\).
Then
\[
 w^2=z^4-2az^2+1.
\]
Set \(U=(w+1)/z^2\), so \(w=Uz^2-1\).  Substitution gives
\[
 (U^2-1)z^2=2(U-a).
\]
With \(V=z(U^2-1)\), this becomes
\[
 V^2=2(U-a)(U^2-1)
 =2(U-a)(U-1)(U+1).
\]
Conversely,
\[
 z=\frac{V}{U^2-1},\qquad w=Uz^2-1
\]
away from finitely many points.  Thus the map is the degree-two
quotient by \(q\mapsto-q\).  All three rational maps extend uniquely
to the smooth projective models.
\end{proof}

\begin{remark}[The dihedral symmetry]
The involution
\[
 \tau[Z:Q:W]=[W:Q:Z]
\]
is defined over \(\Q\).  In projective coordinates, with
\[
 \iota_z[Z:Q:W]=[-Z:Q:W],\qquad
 \iota_q[Z:Q:W]=[Z:-Q:W],
\]
one has
\[
 \tau\iota_z\tau=\iota_z\iota_q,\qquad
 \tau\iota_q\tau=\iota_q.
\]
These automorphisms generate a dihedral subgroup of order \(8\).
On the affine chart,
\[
 \tau(z,q)=\left(\frac1z,\frac qz\right),
 \qquad \phi_2=\phi_1\circ\tau.
\]
This symmetry explains the repeated \(E_b\)-factor in the Jacobian
decomposition.
\end{remark}

\begin{remark}
Each involution has four geometric fixed points:
\[
\begin{array}{c|c}
\iota_z & z=0,\ q^4=1\\
\iota_q & q=0,\ z^4-2az^2+1=0\\
\iota_z\iota_q & \text{the four points at infinity}.
\end{array}
\]
Riemann--Hurwitz again gives genus one for every quotient.  The full
group quotient is the genus-zero conic
\[
 T^2=S^2-2aS+1,\qquad S=z^2,\quad T=q^2,
\]
which contains \((S,T)=(0,1)\).
\end{remark}

\begin{theorem}[Jacobian decomposition]
\label{thm:jacobian-decomposition}
There is an isogeny over \(\Q\)
\begin{equation}\label{eq:jacobian-isogeny}
 {\Jac(C_{a,b})\sim_{\Q}E_b^2\times F_a.}
\end{equation}
\end{theorem}

\begin{proof}
The two genus-one curves \(\mathcal Q_b\) are birational over \(\Q\)
to \(E_b\) by
Theorem~\ref{thm:natural-elliptic-correspondence}; choosing a rational
point at infinity as origin identifies each Jacobian with \(E_b\).
The Klein four-group decomposition of Kani and Rosen
\cite{KaniRosen} therefore gives \eqref{eq:jacobian-isogeny}.

For completeness, the isogeny can also be seen directly on
differentials.  On the affine chart of
\eqref{eq:genus-three-projective}, adjunction gives a basis
\[
 \omega_0=\frac{dz}{q^3},\qquad
 \omega_1=\frac{z\,dz}{q^3},\qquad
 \omega_2=\frac{dz}{q^2}
\]
of \(H^0(C_{a,b},\Omega^1)\).  The actions are
\[
\begin{array}{c|ccc}
 &\omega_0&\omega_1&\omega_2\\
\hline
\iota_z&-&+&-\\
\iota_q&-&-&+\\
\iota_z\iota_q&+&-&-
\end{array}
\]
Thus the invariant eigenspaces belonging to the three quotient maps
are independent one-dimensional subspaces whose direct sum is all of
\(H^0(C_{a,b},\Omega^1)\).  The product of the three norm maps defines
\[
 \Jac(C_{a,b})\longrightarrow E_b\times F_a\times E_b.
\]
The pullback on regular differentials is the direct sum of the three
curve-level pullbacks just described, hence is an isomorphism.  The
homomorphism therefore has finite kernel; since both sides have
dimension \(3\), it is an isogeny over \(\Q\).
\end{proof}

\subsection{Generic rigidity and two elliptic fibrations on a singular
K3 surface}

Put
\[
 K_0=\Q(k),\qquad
 K_{\mathrm{geom}}=\overline{\Q}(k)
 \subset\overline{K_0},\qquad
 a=\frac{2k}{k^2+1},\qquad
 b=\frac{k^2-1}{k^2+1}.
\]
Here and below, \(\overline{K_0}\) is an algebraic closure of \(K_0\).
Thus \(K_{\mathrm{geom}}\) means constant-field extension, not
algebraic closure of the function field.

\begin{theorem}[The generic third quotient]
\label{thm:generic-third-quotient}
Let
\[
 F_a:\quad V^2=2(U-a)(U-1)(U+1)
\]
 over \(K_0\).  The relatively minimal elliptic surface over
\(\PP^1_k\) with generic fiber \(F_a\) is an extremal K3 surface.
Over \(\overline{\Q}\), its singular fibers are
\[
 I_4,\ I_4,\ I_2^*,\ I_2^*
\]
 at \(k=-1,1,i,-i\), respectively.  More precisely,
\begin{equation}\label{eq:generic-F-group}
 {
 F_a(K_{\mathrm{geom}})
 =F_a(K_0)
 =
 \{O,(a,0),(1,0),(-1,0)\}
 \simeq(\Z/2\Z)^2.}
\end{equation}
\end{theorem}

\begin{proof}
Set
\[
 X=\frac{U+1}{2},\qquad
 Y=\frac V4,\qquad
 \lambda=\frac{1+a}{2}
 =\frac{(k+1)^2}{2(k^2+1)}.
\]
Then
\begin{equation}\label{eq:generic-F-Legendre}
 Y^2=X(X-1)(X-\lambda).
\end{equation}
Thus the required surface is the pullback of the Legendre elliptic
surface under the degree-two map
\[
 \beta:\PP^1_k\longrightarrow\PP^1_\lambda,\qquad
 k\longmapsto\frac{(k+1)^2}{2(k^2+1)}.
\]

The Legendre surface has fibers \(I_2,I_2,I_2^*\) at
\(\lambda=0,1,\infty\).  Indeed,
\[
 \Delta=16\lambda^2(1-\lambda)^2,\qquad
 c_4=16(\lambda^2-\lambda+1),
\]
which gives \(I_2\) at \(0\) and \(1\).  At infinity put
\(t=1/\lambda\) and
\[
 X_1=t^2X,\qquad Y_1=t^3Y.
\]
The local equation becomes
\[
 Y_1^2=X_1^3-t(1+t)X_1^2+t^3X_1,
\]
with
\[
 \ord_t(c_4)=2,\qquad \ord_t(\Delta)=8;
\]
Tate's algorithm therefore gives \(I_2^*\).

Now
\[
 \frac{d\lambda}{dk}
 =\frac{1-k^2}{(k^2+1)^2}.
\]
Hence \(\beta\) is ramified with index \(2\) at \(k=-1\) and
\(k=1\), lying over \(0\) and \(1\), respectively.  It is unramified
over infinity, whose two inverse images are \(k=i\) and \(k=-i\).
After minimalization, the pulled-back fibers are therefore \(I_4\)
at each of \(k=-1,1\) and \(I_2^*\) at each of \(k=i,-i\), with no
other singular fibers.

Their Euler numbers add up to
\[
 4+4+8+8=24,
\]
so \(\chi(\mathcal O_{\mathcal F})=24/12=2\).  Since the fibration
has a section, it has no multiple fibers.  The canonical bundle
formula for an elliptic surface over \(\PP^1\) gives
\[
 K_{\mathcal F}
 \simeq
 \pi^*\mathcal O_{\PP^1}
 \bigl(\chi(\mathcal O_{\mathcal F})-2\bigr)
 \simeq\mathcal O_{\mathcal F},
\]
and \(q(\mathcal F)=0\).  Hence \(\mathcal F\) is a K3 surface.
These standard elliptic-surface facts, including base change and the
Shioda--Tate formula used below, may be found in
\cite{SchuettShioda}.

Over \(\overline{\Q}\), its trivial lattice is
\[
 U\oplus A_3^{\,2}\oplus D_6^{\,2},
\]
of rank
\[
 2+3+3+6+6=20.
\]
A complex K3 surface has Picard number at most \(20\).  Therefore
\(\rho(\mathcal F)=20\), and the Shioda--Tate formula gives
\[
 \rank F_a(K_{\mathrm{geom}})
 =\rho(\mathcal F)-20=0.
\]

It remains to determine geometric torsion.  Since the geometric rank
is zero, every point of \(F_a(K_{\mathrm{geom}})\) is torsion.  Choose
the \(I_2^*\)-place \(v=(k-i)\) and pass to
\(K_v\simeq\overline{\Q}((t))\).  The N\'eron filtration is
\[
 F_{a,1}(K_v)\subset F_{a,0}(K_v)\subset F_a(K_v),
\]
with
\[
 F_{a,0}(K_v)/F_{a,1}(K_v)
 \simeq\mathbb G_a(\overline{\Q}),\qquad
 F_a(K_v)/F_{a,0}(K_v)\simeq\Phi_v.
\]
Because the residue characteristic is zero, the formal group
\(F_{a,1}(K_v)\) and the displayed additive quotient have no nonzero
torsion.  The standard N\'eron specialization sequence therefore gives
\[
 F_a(K_v)_{\rm tors}\hookrightarrow\Phi_v.
\]
See \cite[Chapter~VII, \S\S2--3]{Silverman}.
For a fiber of type \(I_2^*\), the component group is
\(\Phi_v\simeq(\Z/2\Z)^2\)
\cite[\S\S4,7]{SchuettShioda}.  The embedding
\(F_a(K_{\mathrm{geom}})\hookrightarrow F_a(K_v)\) now shows that
every geometric torsion point has order dividing \(2\), and that there
are at most four of them.
The three roots of the cubic already give the full group of four
\(2\)-torsion points displayed in \eqref{eq:generic-F-group}.  This
proves both equalities there.
\end{proof}

\begin{proposition}[One geometric K3 surface, two elliptic fibrations]
\label{prop:two-k3-fibrations}
Let \(\mathcal S\) be the Kummer K3 surface of
Proposition~\ref{prop:kummer-k3-surface}, and let \(\mathcal F\) be the
extremal K3 surface of
Theorem~\ref{thm:generic-third-quotient}.  Then
\begin{equation}\label{eq:common-transcendental-lattice}
 T(\mathcal S_{\C})\simeq T(\mathcal F_{\C})
 \simeq\langle4\rangle\oplus\langle4\rangle.
\end{equation}
Consequently,
\begin{equation}\label{eq:geometric-k3-isomorphism}
 {\mathcal S_{\overline{\Q}}\simeq
        \mathcal F_{\overline{\Q}}}
\end{equation}
as abstract K3 surfaces.  The two displayed Jacobian elliptic
fibrations are inequivalent: their singular-fiber configurations and
geometric Mordell--Weil groups are, respectively,
\[
 \begin{array}{c|c}
 4I_0^*&\Z^2\oplus(\Z/2\Z)^2\\
 2I_4+2I_2^*&(\Z/2\Z)^2.
 \end{array}
\]
\end{proposition}

\begin{proof}
By Proposition~\ref{prop:kummer-k3-surface},
\[
 \mathcal S\simeq_{\Q}\operatorname{Km}(A),\qquad
 A=E_0\times E_0,\qquad E_0:y^2=x^3+x.
\]
On \(A_{\C}\), consider the two factor curves, the diagonal
\(\Gamma_1\), and the graph \(\Gamma_i\) of complex multiplication by
\(i\).  Their intersection matrix is
\[
 \begin{pmatrix}
 0&1&1&1\\
 1&0&1&1\\
 1&1&0&2\\
 1&1&2&0
 \end{pmatrix},
\]
whose determinant is \(-4\); here
\(\Gamma_1\cdot\Gamma_i=\deg(1-i)=2\).
Since
\[
 \rho(A_{\C})
 =2+\rank_{\Z}\operatorname{End}_{\C}(E_0)=4,
\]
these four classes have full rank.  Let \(L\) be the lattice that they
generate.  Since
\[
 \operatorname{disc}L
 =[\operatorname{NS}(A_{\C}):L]^2
   \operatorname{disc}\operatorname{NS}(A_{\C})
\]
and \(\operatorname{disc}L=-4\), the index is \(1\) or \(2\).  In the
second case \(\operatorname{NS}(A_{\C})\) would be unimodular.  It is
even, since adjunction on the abelian surface gives
\(D^2=2p_a(D)-2\), and has signature \((1,3)\).  This is impossible
because the signature difference of an even unimodular lattice is
divisible by \(8\).  Hence \(L=\operatorname{NS}(A_{\C})\) and
\(\operatorname{disc}\operatorname{NS}(A_{\C})=-4\).
Now \(H^2(A_{\C},\Z)\simeq U^{\oplus3}\) is even and unimodular, and
\(\operatorname{NS}(A_{\C})\) is primitive in it.  Therefore
\(T(A_{\C})\) is an even positive-definite lattice of rank \(2\) and
determinant \(4\).  By Gauss reduction, the only such lattice is
\(\langle2\rangle\oplus\langle2\rangle\).  Thus
\[
 T(A_{\C})\simeq\langle2\rangle\oplus\langle2\rangle.
\]
The Kummer relation
\[
 T(\operatorname{Km}(A)_{\C})\simeq T(A_{\C})(2)
\]
is an isomorphism of integral Hodge lattices
\cite[\S\S11--12]{SchuettLectures}; here \(L(2)\) means that the
bilinear form of \(L\) is multiplied by \(2\).  It now gives
\[
 T(\mathcal S_{\C})\simeq
 \langle4\rangle\oplus\langle4\rangle
\]
as claimed.

For \(\mathcal F\), the reducible-fiber root lattice is
\[
 2A_3+2D_6.
\]
Entry No.~137 of the extremal elliptic K3 classification of Shimada
and Zhang \cite[Table~2, No.~137]{ShimadaZhang} records for this
configuration the geometric Mordell--Weil
group \((\Z/2\Z)^2\) and the transcendental Gram matrix
\[
 \begin{pmatrix}4&0\\0&4\end{pmatrix}.
\]
This proves \eqref{eq:common-transcendental-lattice}.  The
discriminant check is
\[
 \left|\operatorname{disc}\operatorname{NS}(\mathcal F_{\C})\right|
 =\frac{4^2\,4^2}{\left|(\Z/2\Z)^2\right|^2}=16.
\]

Singular complex K3 surfaces---that is, K3 surfaces of Picard number
\(20\)---are classified by their naturally oriented positive-definite
even transcendental lattices
\cite{ShimadaZhang,ShiodaInose}.  The lattice
\(\langle4\rangle\oplus\langle4\rangle\) has an
orientation-reversing isometry, so no orientation ambiguity remains.
Thus \(\mathcal S_{\C}\simeq\mathcal F_{\C}\).  Let
\[
 I=\underline{\operatorname{Isom}}_{\overline{\Q}}
   (\mathcal S_{\overline{\Q}},\mathcal F_{\overline{\Q}}).
\]
For projective varieties, \(I\) is represented by a scheme locally
of finite type over \(\overline{\Q}\), and its formation
commutes with base change.  Since \(I_{\C}\ne\varnothing\), faithful
flatness of \(\C/\overline{\Q}\) gives \(I\ne\varnothing\); because
\(\overline{\Q}\) is
algebraically closed, \(I\) has an \(\overline{\Q}\)-point.
This proves \eqref{eq:geometric-k3-isomorphism}.

Finally, an isomorphism intertwining the two elliptic fibrations, even
after an automorphism of the base, would preserve the multiset of
Kodaira fiber types.  The configurations \(4I_0^*\) and
\(2I_4+2I_2^*\) differ, so no such isomorphism exists.
\end{proof}

\begin{corollary}\label{cor:common-k3-ns}
Both underlying geometric K3 surfaces satisfy
\[
 \rho(\mathcal S_{\overline{\Q}})
 =\rho(\mathcal F_{\overline{\Q}})=20,\qquad
 \left|\operatorname{disc}
 \operatorname{NS}(\mathcal S_{\overline{\Q}})\right|
 =
 \left|\operatorname{disc}
 \operatorname{NS}(\mathcal F_{\overline{\Q}})\right|=16.
\]
\end{corollary}

\begin{proof}
Fix an embedding \(\overline{\Q}\hookrightarrow\C\).  Picard number and
the N\'eron--Severi intersection lattice are unchanged by this
algebraically closed extension.  For either
\(X=\mathcal S\) or \(X=\mathcal F\), the lattice
\(H^2(X_{\C},\Z)\) is even unimodular of rank \(22\), and
\(\operatorname{NS}(X_{\C})\) is primitive with orthogonal complement
\(T(X_{\C})\).  Equation~\eqref{eq:common-transcendental-lattice}
therefore gives
\[
 \rho(X_{\overline{\Q}})=22-2=20.
\]
The discriminant groups of \(\operatorname{NS}(X_{\C})\) and
\(T(X_{\C})\) have the same order.  Since
\(\det T(X_{\C})=16\), the asserted N\'eron--Severi discriminants
follow.
\end{proof}

\begin{corollary}[Generic Jacobian ranks]
\label{cor:generic-jacobian-ranks}
Let \(C_{\rm gen}=C_{a(k),b(k)}\) and
\(J_{\rm gen}=\Jac(C_{\rm gen})\).  Then
\[
 {\rank J_{\rm gen}(K_0)=2},\qquad
 {\rank J_{\rm gen}(K_{\mathrm{geom}})=4}.
\]
\end{corollary}

\begin{proof}
The isogeny
\[
 J_{\rm gen}\sim E_b^2\times F_a
\]
preserves Mordell--Weil rank.  The two-square quotient \(E_b\) has
rank \(1\) over \(K_0\) by
Theorem~\ref{thm:generic-mordell-weil}, and rank \(2\) over
\(K_{\mathrm{geom}}\) by Remark~\ref{rem:geometric-generic-rank}.
Theorem~\ref{thm:generic-third-quotient} gives geometric rank zero for
\(F_a\).  The two displayed ranks follow.
\end{proof}

\begin{proposition}[Generic endomorphisms and N\'eron--Severi ranks]
\label{prop:generic-endomorphism-ns}
Let \(J_{\rm gen}=\Jac(C_{\rm gen})\), and put
\[
 \rho_{K_0}(J_{\rm gen})
 =
 \dim_{\Q}
 \left(
  \operatorname{NS}((J_{\rm gen})_{\overline{K_0}})
  \otimes_{\Z}\Q
 \right)^{G_{K_0}},
\qquad
 G_{K_0}=\operatorname{Gal}(\overline{K_0}/K_0).
\]
Write \(\rho_{\overline{K_0}}(J_{\rm gen})\) for the corresponding
dimension without taking Galois invariants.
Then
\begin{align}
 \operatorname{End}^0_{K_0}(J_{\rm gen})
 &\simeq M_2(\Q)\times\Q,
 \label{eq:generic-end-K}\\
 \operatorname{End}^0_{\overline{K_0}}(J_{\rm gen})
 &\simeq M_2(\Q(i))\times\Q,
 \label{eq:generic-end-Kbar}
\end{align}
and
\begin{equation}\label{eq:generic-ns-ranks}
 {\rho_{K_0}(J_{\rm gen})=4,\qquad
 \rho_{\overline{K_0}}(J_{\rm gen})=5.}
\end{equation}
\end{proposition}

\begin{proof}
The quotient maps give a \(K_0\)-isogeny
\[
 J_{\rm gen}\sim_{K_0} E_b^2\times F_a.
\]
Here \(j(E_b)=1728\), whereas direct substitution in the
\(j\)-invariant of \(F_a\) gives
\begin{equation}\label{eq:generic-F-j}
 j(F_a)=
 64\frac{(3k^4+10k^2+3)^3}
 {(k^2+1)^2(k^2-1)^4},
\end{equation}
which is nonconstant.  Thus \(E_b\) is isotrivial and \(F_a\) is
nonisotrivial.  An elliptic curve in characteristic zero whose
geometric endomorphism algebra is larger than \(\Q\) has complex
multiplication and hence constant \(j\)-invariant.  Therefore
\[
 \operatorname{End}^0_{\overline{K_0}}(F_a)
 =\operatorname{End}^0_{K_0}(F_a)=\Q.
\]
Moreover, isotriviality is preserved under isogeny, so \(E_b\) and
\(F_a\) are not geometrically isogenous and the cross-Hom groups
between them vanish.

Geometrically \(E_b\) is a twist of \(E_0:y^2=x^3+x\), whence
\[
 \operatorname{End}^0_{\overline{K_0}}(E_b)=\Q(i).
\]
The nontrivial CM endomorphism is represented after trivializing the
twist by \((x,y)\mapsto(-x,iy)\).  It is not defined over
\(K_0=\Q(k)\) \cite[Chapter~III, \S9]{Silverman}, and therefore
\[
 \operatorname{End}^0_{K_0}(E_b)=\Q.
\]
The vanishing of the cross-Hom groups and invariance of rational
endomorphism algebras under isogeny now give
\eqref{eq:generic-end-K} and \eqref{eq:generic-end-Kbar}.

Put \(B=E_b^2\times F_a\) and give it the product polarization.
Pullback by an isogeny induces a \(G_{K_0}\)-equivariant isomorphism of
N\'eron--Severi spaces over \(\Q\).  The standard
N\'eron--Severi--Rosati correspondence identifies the
N\'eron--Severi space of \(B\) with the subspace of
\(\operatorname{End}^0(B)\) fixed by the Rosati involution
\cite[\S21]{Mumford}.  Over \(K_0\), Rosati is transpose on
\(M_2(\Q)\); its fixed space consists of symmetric matrices and has
dimension \(3\).  The \(F_a\)-factor contributes one more dimension.
Thus \(\rho_{K_0}(J_{\rm gen})=4\).

Over \(\overline{K_0}\), Rosati is conjugate transpose on
\(M_2(\Q(i))\).  Its fixed space is
\(\operatorname{Herm}_2(\Q(i))\), of dimension \(4\) over \(\Q\).
Together with the \(F_a\)-factor this gives
\(\rho_{\overline{K_0}}(J_{\rm gen})=5\).
\end{proof}

\begin{remark}
The preceding equalities concern the generic Jacobian.  On a
specialized rational fiber, additional complex multiplication or
additional isogenies between the elliptic factors may increase the
arithmetic or geometric N\'eron--Severi rank.  Consequently, the
fiberwise argument in Section~\ref{sec:examples} uses only the
unconditional lower bound \(\rho_{\Q}(J)\ge4\).
\end{remark}

\begin{theorem}[Absence of nontrivial rational sections]
\label{thm:no-nontrivial-generic-sections}
For the smooth projective generic curve
\[
 C_{\rm gen}:\quad q^4=z^4-2az^2+1
\]
over \(K_0=\Q(k)\), one has
\begin{equation}\label{eq:generic-C-points}
 {C_{\rm gen}(K_0)=\mathcal T,}
\end{equation}
where
\[
 \mathcal T=
 \{[0:1:1],[0:-1:1],[1:1:0],[1:-1:0]\}.
\]
Equivalently, the all-square genus-three fibration has no nontrivial
rational section over the \(k\)-line.
\end{theorem}

\begin{proof}
The only \(K_0\)-rational points with \(ZW=0\) are the four displayed
points, because the fourth roots of unity in \(K_0\) are \(1\) and
\(-1\).  There is no \(K_0\)-point with \(Q=0\): in affine coordinates
this would give
\[
 (z^2-a)^2+b^2=0,
\]
which is impossible because \(K_0=\Q(k)\) is formally real and
\(b\ne0\).

Suppose that \(P=(z,q)\) is a point outside \(\mathcal T\).  Then
\(zq\ne0\), so its image under the third quotient is the affine point
\[
 \phi_3(P)=(U,V)\in F_a(K_0),\qquad
 U=\frac{q^2+1}{z^2},\qquad V=z(U^2-1).
\]
Eliminating \(q\) from the defining equation gives
\begin{equation}\label{eq:generic-C-lift-relation}
 {(U^2-1)z^2=2(U-a).}
\end{equation}
By Theorem~\ref{thm:generic-third-quotient}, the \(U\)-coordinate of
\(\phi_3(P)\) must be \(a,1\), or \(-1\), unless the image is \(O\).
The value \(U=a\) would give
\[
 (a^2-1)z^2=-b^2z^2=0,
\]
contrary to \(bz\ne0\).  The values \(U=1\) and \(U=-1\) would give
\(1-a=0\) and \(1+a=0\), respectively, also impossible in \(K_0\).
Finally, \(O\) cannot occur because the formulas for \(\phi_3(P)\)
give an affine image whenever \(z\ne0\).  Thus no point outside
\(\mathcal T\) exists.
\end{proof}

\begin{corollary}[Rank and the classical Chabauty hypothesis]
\label{cor:rank-formula}
Let
\[
 r_b=\rank E_b(\Q),\qquad r_a=\rank F_a(\Q).
\]
Then
\begin{equation}\label{eq:rank-formula}
 {\rank\Jac(C_{a,b})(\Q)=2r_b+r_a.}
\end{equation}
The standard sufficient rank hypothesis for classical
Chabauty--Coleman is
\begin{equation}\label{eq:chabauty-rank-condition}
 {2r_b+r_a<3.}
\end{equation}
For \(k\in\Q\setminus\{0,\pm1\}\), this inequality holds if and only if
\[
 r_b=1,\qquad r_a=0.
\]
\end{corollary}

\begin{proof}
Mordell--Weil rank is invariant under isogeny and additive on direct
products, proving \eqref{eq:rank-formula}.  The classical Chabauty
hypothesis is that the Jacobian rank be strictly less than the genus,
which is \(3\) here \cite{Chabauty,Coleman}.  For every
\(k\in\mathbf Q\setminus\{0,\pm1\}\),
Theorem~\ref{thm:uniform-nontorsion} gives \(r_b\ge1\), so
\eqref{eq:chabauty-rank-condition} has exactly the stated form.
\end{proof}

\begin{remark}\label{rem:coleman-bound}
If \(p>6\) is a prime of good reduction and the required
Mordell--Weil information has been rigorously verified, Coleman's
bound under \eqref{eq:chabauty-rank-condition} is
\[
 \#C_{a,b}(\Q)\le \#C_{a,b}(\F_p)+4.
\]
Writing \(a_p(E_b)\) and \(a_p(F_a)\) for the Frobenius traces at a
common prime of good reduction, the Jacobian isogeny gives
\[
 \#C_{a,b}(\F_p)
 =p+1-2a_p(E_b)-a_p(F_a).
\]
Thus the same bound may be written
\[
 \#C_{a,b}(\Q)
 \le p+5-2a_p(E_b)-a_p(F_a).
\]
As unpolarized abelian varieties, one has the \(\Q\)-isogeny
\[
 \Prym(C_{a,b}/\mathcal Q_b)\sim_{\Q}E_b\times F_a
\]
and, for the Hasse--Weil \(L\)-functions, the decomposition gives
\[
 L(C_{a,b},s)=L(E_b,s)^2L(F_a,s).
\]
\end{remark}

\subsection{Local obstructions for primitive all-square triangles}

\begin{proposition}\label{prop:all-square-local-obstructions}
Let
\[
 (X^2,Y^2,Z^2)
\]
be a primitive all-square Heron triangle, and let its area be
\(\Delta\in\Z_{>0}\).  Then \(\gcd(X,Y,Z)=1\) and:
\begin{enumerate}
\setlength{\itemsep}{0.45\baselineskip}
\item exactly one of \(X,Y,Z\) is even.  If this root is denoted by
\(X_e\), then
\[
 \ord_2(X_e)=1\Longrightarrow\ord_2(\Delta)=1,
 \qquad
 4\mid X_e\Longrightarrow4\mid\Delta;
\]
\item neither exactly two nor all three of \(X,Y,Z\) are divisible by
\(3\), and, up to permutation,
\[
 \{X^2,Y^2,Z^2\}\pmod9
 \in
 \left\{
 \begin{array}{c}
 \{0,1,1\},\{0,4,4\},\{0,7,7\},\\
 \{1,1,7\},\{1,4,4\},\{4,7,7\}
 \end{array}
 \right\};
\]
\item if exactly one of \(X,Y,Z\) is divisible by \(3\), then
\(9\mid\Delta\), and hence \(18\mid\Delta\);
\item if \(3\nmid XYZ\), then \(3\mid\Delta\) and
\[
 X^2+Y^2+Z^2\equiv0\pmod9,
\]
\item in all cases \(6\mid\Delta\).
\end{enumerate}
\end{proposition}

\begin{proof}
Primitivity of the side triple gives
\(\gcd(X^2,Y^2,Z^2)=1\), and hence \(\gcd(X,Y,Z)=1\).
Heron's identity is
\[
 16\Delta^2=S_0S_1S_2S_3,
\]
where
\[
\begin{aligned}
 S_0&=X^2+Y^2+Z^2,&
 S_1&=-X^2+Y^2+Z^2,\\
 S_2&=X^2-Y^2+Z^2,&
 S_3&=X^2+Y^2-Z^2.
\end{aligned}
\]
If \(X,Y,Z\) were all odd, or if exactly two were even, all four
\(S_i\) would be odd, which is incompatible with the factor \(16\)
on the left.  They cannot all be even because the triangle is
primitive.  Hence exactly one is even.

Suppose \(X\) is even and \(Y,Z\) are odd.  Then \(S_0,S_1\) are
of exact \(2\)-adic valuation \(1\).  If \(\ord_2(X)=1\), then
\(X^2\equiv4\pmod8\), while the difference of two odd squares is
divisible by \(8\); hence
\[
 \ord_2(S_2)=\ord_2(S_3)=2.
\]
Thus
\[
 4+2\ord_2(\Delta)
 =\ord_2(S_0S_1S_2S_3)=1+1+2+2=6,
\]
so \(\ord_2(\Delta)=1\).  If \(4\mid X\), then
\(\ord_2(S_2),\ord_2(S_3)\ge3\), and the same identity gives
\(\ord_2(\Delta)\ge2\).  In particular \(2\mid\Delta\) in every case.

If exactly two roots, say \(X,Y\), were divisible by \(3\), then
\[
 (S_0,S_1,S_2,S_3)\equiv(1,1,1,-1)\pmod3,
\]
so \(16\Delta^2\equiv-1\pmod3\), impossible.  If exactly one root is
divisible by \(3\), say \(X\), then \(9\mid X^2\) and
\(3\mid S_2,S_3\).  We claim that
\[
 Y^2\equiv Z^2\pmod9.
\]
Otherwise \(S_2/3\) and \(S_3/3\) are nonzero and opposite modulo
\(3\), while \(S_0S_1\equiv1\pmod3\).  Since
\(3\mid S_2,S_3\), Heron's identity first gives \(3\mid\Delta\).
Dividing that identity by \(9\) would then say that a square is
congruent to \(-1\pmod3\), a contradiction.  Thus
\(9\mid S_2,S_3\), whence
\[
 \ord_3(16\Delta^2)\ge4,\qquad 9\mid\Delta.
\]
The possible residues in this case are precisely
\(\{0,1,1\},\{0,4,4\},\{0,7,7\}\).
All three roots cannot be divisible by \(3\), since then all three
side lengths would have a common factor \(9\), contrary to
primitivity.

Finally suppose \(3\nmid XYZ\).  Then all three squares are \(1\)
modulo \(3\).  Thus \(3\mid S_0\), while none of \(S_1,S_2,S_3\) is
divisible by \(3\).  Consequently
\[
 \ord_3(S_0)=\ord_3(16\Delta^2)=2\ord_3(\Delta)
\]
is a positive even number.  Therefore \(9\mid S_0\) and
\(3\mid\Delta\).  The nonzero square classes modulo \(9\) are
\(1,4,7\); enumerating the unordered triples whose sum is \(0\)
 modulo \(9\) gives exactly the three triples in the statement.
 Together with the case of one root divisible by \(3\), this gives
 the six residue triples in the statement.  Since \(2\mid\Delta\) in
 all cases and \(3\mid\Delta\) in both possible \(3\)-adic cases, it
 also proves \(6\mid\Delta\); in the one-divisible-root case the
 stronger conclusion is \(18\mid\Delta\).
\end{proof}

\begin{proposition}[Further local restrictions]
\label{prop:further-all-square-local-obstructions}
Let \((X^2,Y^2,Z^2)\) be a primitive all-square Heron triangle of
area \(\Delta\).
\begin{enumerate}
\item For every prime \(p\equiv3\pmod4\), at most one of \(X,Y,Z\)
is divisible by \(p\).  If \(p\mid X\), then
\[
 Y^2\equiv Z^2\pmod p,\qquad p\mid\Delta.
\]
\item One has
\begin{equation}\label{eq:mod5-root-obstruction}
 {5\mid XYZ.}
\end{equation}
Equivalently, at least one side of the triangle is divisible by
\(25\).  If exactly one of \(X,Y,Z\) is divisible by \(5\), then
\(30\mid\Delta\); if exactly two are divisible by \(5\), then
\(5\nmid\Delta\).
\end{enumerate}
\end{proposition}

\begin{proof}
The expanded Heron identity is
\begin{equation}\label{eq:expanded-square-heron}
\begin{split}
 16\Delta^2={}&
 2X^4Y^4+2X^4Z^4+2Y^4Z^4\\
 &{}-X^8-Y^8-Z^8.
\end{split}
\end{equation}
Suppose \(p\equiv3\pmod4\).  If \(p\) divided, say, \(X\) and \(Y\)
but not \(Z\), then \eqref{eq:expanded-square-heron} would give
\[
 (4\Delta)^2\equiv-Z^8\pmod p,
\]
which is impossible because \(-1\) is a nonsquare modulo \(p\).
Primitivity excludes divisibility of all three roots, proving the
first assertion.  If \(p\mid X\), then \(p\nmid YZ\) and
\[
 (4\Delta)^2\equiv-(Y^4-Z^4)^2\pmod p.
\]
Again the nonsquareness of \(-1\) forces both sides to vanish.  Hence
\(p\mid\Delta\) and \(Y^4\equiv Z^4\pmod p\).  Since
\((Y/Z)^2\) cannot be \(-1\), one obtains
\(Y^2\equiv Z^2\pmod p\).

Modulo \(5\), every fourth power is \(0\) or \(1\).  If
\(5\nmid XYZ\), the right side of
\eqref{eq:expanded-square-heron} is congruent to \(3\), a nonsquare
modulo \(5\), which proves \eqref{eq:mod5-root-obstruction}.
Primitivity rules out divisibility of all three roots.  If exactly one
root is divisible by \(5\), the right side of
\eqref{eq:expanded-square-heron} is \(0\) modulo \(5\), so
\(5\mid\Delta\); Proposition~\ref{prop:all-square-local-obstructions}
also gives \(6\mid\Delta\), hence \(30\mid\Delta\).  If exactly two
roots are divisible by \(5\), the right side is
\(-1\equiv4\pmod5\), so \(5\nmid\Delta\).
\end{proof}

\begin{remark}
Proposition~\ref{prop:all-square-local-obstructions} does not prove
that exactly one of \(X,Y,Z\) is divisible by \(3\).  The case
\(3\nmid XYZ\) remains possible, subject to the stated modulo \(9\)
restriction.
\end{remark}

\subsection{The right and isosceles all-square cases}

\begin{lemma}[Fermat's quartic descent]\label{lem:fermat-quartic}
There are no positive integers \(A,B,C\) satisfying
\[
 A^4+B^4=C^2.
\]
\end{lemma}

\begin{proof}
This is Fermat's classical infinite-descent theorem; see
\cite[Section~6.5.2, Proposition~6.5.3]{Cohen2007},
with \(\varepsilon=1\).
\end{proof}

\begin{proposition}[The isosceles genus-one bisection]
\label{prop:isosceles-genus-one-bisection}
On the total all-square fibration over the \(k\)-line, the locus
\(q=1\) is the union of the trivial section \(z=0\) and the bisection
\begin{equation}\label{eq:isosceles-bisection}
 \mathcal B:\qquad z^2(k^2+1)=4k.
\end{equation}
The smooth projective model of \(\mathcal B\) is
\[
 E_0:\qquad y^2=x^3+x
\]
under
\begin{equation}\label{eq:isosceles-bisection-map}
 x=k,\qquad y=\frac{z(k^2+1)}2,
 \qquad z=\frac{2y}{x^2+1}.
\end{equation}
Moreover,
\begin{equation}\label{eq:E0-rational-points}
 {E_0(\Q)=\{O,(0,0)\}.}
\end{equation}
Consequently, this bisection has no rational point producing a
nondegenerate all-square triangle.
\end{proposition}

\begin{proof}
Substituting \(q=1\) and \(a=2k/(k^2+1)\) into
\[
 q^4=z^4-2az^2+1
\]
gives
\[
 z^2(z^2-2a)=0.
\]
The nontrivial component is \eqref{eq:isosceles-bisection}, and
\eqref{eq:isosceles-bisection-map} gives
\[
 y^2=k(k^2+1)=x^3+x.
\]
The projection to the \(k\)-line is the degree-two \(x\)-map.

It remains to determine \(E_0(\Q)\).  A rational point with \(x<0\)
is impossible over \(\R\).  If \(x>0\), the usual denominator
valuation argument writes it in coprime form as
\[
 x=\frac{A}{C^2},\qquad y=\frac{B}{C^3},
 \qquad A,C>0,\quad\gcd(A,C)=1.
\]
The equation becomes
\[
 B^2=A(A^2+C^4).
\]
The two positive factors on the right are coprime, so
\[
 A=r^2,\qquad A^2+C^4=s^2
\]
for positive integers \(r,s\).  This gives
\[
 r^4+C^4=s^2,
\]
contrary to Lemma~\ref{lem:fermat-quartic}.  Thus
\(x=0\), giving \((0,0)\), besides the point \(O\) at infinity.  Both
lie over the boundary \(z=0\), so neither yields a nondegenerate
triangle.
\end{proof}

\begin{corollary}\label{cor:a-zero-trivial-points}
For \(a=0\) and \(b=\pm1\),
\[
 C_{0,\pm1}(\Q)=\mathcal T.
\]
Thus the fiber with right oriented angle parameter produces no
nondegenerate all-square triangle.
\end{corollary}

\begin{proof}
The projective equation is
\[
 Q^4=Z^4+W^4.
\]
If \(ZW\ne0\), clearing denominators would give a positive integral
solution of \(A^4+B^4=C^4\), contrary to
Lemma~\ref{lem:fermat-quartic}.  The cases \(Z=0\) and
\(W=0\) give exactly the four points of \(\mathcal T\); the case
\(Q=0\) gives no projective point over \(\Q\).
\end{proof}

\begin{corollary}\label{cor:no-right-isosceles-all-square}
There is no nondegenerate right or isosceles rational Heron triangle
whose three sides are rational squares.
\end{corollary}

\begin{proof}
For a right triangle, clearing square denominators gives positive
integers satisfying
\[
 A^4+B^4=C^4,
\]
which is a special case of
Lemma~\ref{lem:fermat-quartic}.

For an isosceles all-square triangle, mark the equal square sides and
normalize them to \(1\).  It would then give a nonboundary rational
point on the bisection \(q=1\), contrary to
Proposition~\ref{prop:isosceles-genus-one-bisection}.
\end{proof}

\begin{theorem}[Counting primitive isosceles triangles]
\label{thm:isosceles-counting}
Let \(N_{\rm iso}(X)\) be the number of unmarked primitive isosceles
Heron triangles whose equal sides are integer squares and whose
largest side is at most \(X\).  Then
\[
 {N_{\rm iso}(X)\asymp X^{1/2}.}
\]
Every triangle counted here has exactly two square sides.
\end{theorem}

\begin{proof}
By Corollary~\ref{cor:two-isosceles-families}, every such triangle
comes from one of the two displayed families, with
\[
 u>v>0,\qquad \gcd(u,v)=1,\qquad u\not\equiv v\pmod2,
\]
and its equal side is
\[
 D=(u^2+v^2)^2.
\]
If the largest side is at most \(X\), then \(D\le X\), so
\(u,v=O(X^{1/4})\).  The two families therefore give
\[
 N_{\rm iso}(X)\ll X^{1/2}.
\]

For the reverse bound use only
\eqref{eq:isosceles-family-one} and restrict to
\[
 1\le v\le B,\qquad 2v<u<3v.
\]
Writing \(u=2v+j\), the coprime opposite-parity choices number
\(\varphi(v)\) when \(v\) is even and \(\varphi(v)/2\) when \(v>1\)
is odd.  Hence their total number is
\[
 \#\{\text{such pairs}\}
 \gg\sum_{v\le B}\varphi(v)\gg B^2.
\]
For these pairs,
\[
 D<100B^4,
\]
and the third side is less than \(2D\) by the triangle inequality.
Thus the largest side is less than \(200B^4\).

Distinct parameter pairs in this subfamily give distinct unmarked
triangles.  Indeed, the equal side recovers the hypotenuse
\(u^2+v^2\), while one quarter of the third side recovers the product
of the two legs \(u^2-v^2\) and \(2uv\); their unordered pair, and then
the primitive Pythagorean parameter \(u>v\), is unique.  Taking
\(B=\lfloor(X/200)^{1/4}\rfloor\) proves the lower bound.  Finally,
Corollary~\ref{cor:no-right-isosceles-all-square} excludes a third
square side.
\end{proof}

\Needspace{18\baselineskip}
\section{The global all-square surface and conditional sparsity}
\label{sec:global-all-square-surface}

For an all-square triangle with side lengths
\((X^2,Y^2,Z^2)\) and area \(\Delta\), Heron's identity, with
\(\Omega=4\Delta\), gives the weighted double plane
\begin{equation}\label{eq:global-all-square-surface}
\begin{split}
 \mathscr H:\quad \Omega^2={}&
 (X^2+Y^2+Z^2)(-X^2+Y^2+Z^2)\\
 &{}\times(X^2-Y^2+Z^2)(X^2+Y^2-Z^2)
 \subset\PP(1,1,1,4).
\end{split}
\end{equation}
This surface should not be confused with the singular K3 surface
underlying the two elliptic quotient fibrations: \(\mathscr H\) is
the global all-square surface, birational to the total space of the
genus-three fibration, and is of general type.

\begin{proposition}[The global all-square surface]
\label{prop:global-all-square-surface}
Geometrically, the surface \(\mathscr H\) has exactly twelve singular
points, each a Du Val singularity of type \(A_3\).  In particular,
\(\mathscr H\) is a normal Gorenstein surface.  If
\(\rho:\widetilde{\mathscr H}\to\mathscr H\) is its minimal
resolution, then \(\widetilde{\mathscr H}\) is a minimal surface of
general type and
\[
 K_{\widetilde{\mathscr H}}^2=2,\qquad
 p_g(\widetilde{\mathscr H})=3,\qquad
 q(\widetilde{\mathscr H})=0,\qquad
 \chi(\mathcal O_{\widetilde{\mathscr H}})=4,
 \qquad c_2(\widetilde{\mathscr H})=46.
\]
In particular, \(\kappa(\widetilde{\mathscr H})=2\).
\end{proposition}

\begin{proof}
Let \(B_0,B_1,B_2,B_3\) be the four quadratic factors on the right of
\eqref{eq:global-all-square-surface}.  Each \(B_j=0\) is a smooth
conic, and their pairwise intersections are
\[
\begin{array}{lll}
 B_0\cap B_1=\{[0:1:\pm i]\},&
 B_0\cap B_2=\{[1:0:\pm i]\},&
 B_0\cap B_3=\{[1:\pm i:0]\},\\
 B_1\cap B_2=\{[1:\pm1:0]\},&
 B_1\cap B_3=\{[1:0:\pm1]\},&
 B_2\cap B_3=\{[0:1:\pm1]\}.
\end{array}
\]
 The twelve points are distinct, there is no triple intersection, and
the two conics at each point have intersection multiplicity \(2\).
In suitable completed local coordinates the branch divisor is
\(y(y-x^2)=0\), so the double cover is
\[
 w^2=y(y-x^2).
 \]
 Setting
\[
 u=2y-x^2+2w,\qquad v=2y-x^2-2w
\]
 gives \(uv=x^4\), the \(A_3\) equation.  The surface avoids the
 singular vertex of \(\PP(1,1,1,4)\), and the branch divisor is smooth
 away from these twelve points.  Thus there are no other singularities.
 Since \(\mathscr H\) lies in the smooth locus of
 \(\PP(1,1,1,4)\), it is a Cartier hypersurface and hence
 Cohen--Macaulay and Gorenstein.  Its singular locus is finite, so
 Serre's criterion implies that \(\mathscr H\) is normal.

 Let \(\pi:\mathscr H\to\PP^2\) be the double-cover map and put
 \(\widetilde\pi=\pi\circ\rho\).  The branch divisor has degree \(8\);
 since \(\mathscr H\) is a degree-eight hypersurface in
 \(\PP(1,1,1,4)\) and avoids the singular vertex, weighted adjunction
 gives
 \[
 \omega_{\mathscr H}
 \simeq\mathcal O_{\mathscr H}(8-1-1-1-4)
 =\mathcal O_{\mathscr H}(1)
 =\pi^*\mathcal O_{\PP^2}(1).
 \]
An \(A_3\) singularity is Du Val, so its minimal resolution is
crepant \cite[III.7]{BHPV}.  Consequently,
\[
 K_{\widetilde{\mathscr H}}
 =\widetilde\pi^*\mathcal O_{\PP^2}(1),\qquad
 K_{\widetilde{\mathscr H}}^2=2.
\]
This canonical divisor is nef and big.  A \((-1)\)-curve would have
intersection \(-1\) with it by adjunction, so none exists.  The
resolution is therefore minimal and has Kodaira dimension \(2\).

The double-cover algebra is
\[
 \pi_*\mathcal O_{\mathscr H}
 =\mathcal O_{\PP^2}\oplus\mathcal O_{\PP^2}(-4).
\]
Rational double points are rational singularities
\cite[III.3]{BHPV}, so the resolution does not change the cohomology
of the structure sheaf.  Serre duality on \(\PP^2\) now gives
\[
 q=h^1(\mathcal O_{\PP^2})+h^1(\mathcal O_{\PP^2}(-4))=0,
 \qquad
 p_g=h^2(\mathcal O_{\PP^2}(-4))
 =h^0(\mathcal O_{\PP^2}(1))=3.
\]
Thus \(\chi(\mathcal O_{\widetilde{\mathscr H}})=1-q+p_g=4\), and
Noether's formula gives \(c_2=12\chi-K^2=46\).
\end{proof}

On the open set \(XZ\ne0\), put
\begin{equation}\label{eq:global-surface-fiber-coordinates}
 q=\frac YX,\qquad z=\frac ZX,\qquad
 a=\frac{X^4+Z^4-Y^4}{2X^2Z^2},\qquad
 b=\frac{\Omega}{2X^2Z^2}.
\end{equation}
The equation of \(\mathscr H\) gives \(a^2+b^2=1\), and direct
substitution gives
\[
 q^4=z^4-2az^2+1.
\]
Conversely, a point of this affine quartic gives
\[
 [X:Y:Z:\Omega]=[1:q:z:2bz^2]\in\mathscr H.
\]
Thus, away from the degenerate boundary, \(\mathscr H\) is birational
to the total genus-three fibration.  Under
\[
 a=\frac{2k}{k^2+1},\qquad b=\frac{k^2-1}{k^2+1},
\]
one has \(k=a/(1-b)\) when \(b\ne1\), with \(b=1\) representing
\(k=\infty\).  For positive square side lengths, at most one triangle
inequality can fail, so \(\Omega\ne0\) forces all three inequalities
to be strict.  After deleting \(XYZ\Omega=0\), the orbits of the
remaining rational points under the finite group generated by the sign
changes of the square roots, the involution
\(\Omega\mapsto-\Omega\), and edge permutations are in bijection with
the unmarked rational similarity classes of nondegenerate all-square
Heron triangles.

\begin{corollary}[Conditional thinness]
\label{cor:conditional-all-square-thinness}
Assume the weak Bombieri--Lang conjecture for
\(\widetilde{\mathscr H}\) over \(\Q\), in the form that
\(\widetilde{\mathscr H}(\Q)\) is not Zariski dense
\cite{Bresciani}.  Then
\[
 \mathcal K_{\rm all}
 =
 \left\{k\in\Q\setminus\{0,\pm1\}:
 C_{a(k),b(k)}(\Q)\setminus\mathcal T\ne\varnothing\right\}
\]
is a thin subset of \(\PP^1(\Q)\).

Under the same conjectural assumption, infinitude of primitive
all-square Heron triangles up to edge permutation would imply the
existence of a geometrically integral nontrivial multisection
\(D/\Q\), of degree at least \(2\), such that its smooth projective
normalization \(D^\nu\) satisfies one of the following:
\begin{enumerate}
\item \(D^\nu\) has genus \(0\) and \(D^\nu(\Q)\ne\varnothing\);
\item \(D^\nu\) has genus \(1\), \(D^\nu(\Q)\ne\varnothing\), and
\(\rank\Jac(D^\nu)(\Q)>0\).
\end{enumerate}
\end{corollary}

\begin{proof}
Let
\[
 U=\{XYZ\Omega\ne0\}\subset\mathscr H,\qquad
 \widetilde U=\rho^{-1}(U).
\]
All twelve singular points lie on \(XYZ=0\), so \(\rho\) restricts
to an isomorphism \(\widetilde U\simeq U\).
Formula~\eqref{eq:global-surface-fiber-coordinates} defines a morphism
from \(U\) to the affine conic \(a^2+b^2=1\).  Composing with its
smooth projective completion and the standard identification of that
conic with \(\PP^1_k\) gives a morphism
\[
 f:\widetilde U\longrightarrow\PP^1_k.
\]
Let \(\mathcal Z\) be the Zariski closure of \(\widetilde U(\Q)\) in
\(\widetilde{\mathscr H}\).  By weak Bombieri--Lang, \(\mathcal Z\) is
a proper closed subset, so its one-dimensional part is a finite union
of curves.
Write these as irreducible curves over \(\Q\).  If such a curve is not
geometrically integral, each of its \(\Q\)-points lies in the
intersection of its distinct conjugate geometric components, and
there are only finitely many such points.  Thus, after discarding
finitely many points and components not meeting \(\widetilde U\), we
may retain finitely many geometrically integral curves \(D_i/\Q\).

Recall that a subset of \(\PP^1(\Q)\) is thin if it is contained in
\(Z_0(\Q)\) together with finitely many images \(f_i(D_i(\Q))\), where
\(Z_0\subsetneq\PP^1\) is closed and each
\(f_i:D_i\to\PP^1\) is a dominant generically finite morphism of
degree at least \(2\) from a geometrically integral curve over \(\Q\)
\cite[Definition~3.1.1]{SerreGalois}.

A vertical curve contributes only one parameter.  After replacing
each horizontal \(D_i\) by its smooth projective normalization, the
restriction of \(f\) extends to a morphism of curves.  If its degree
were \(1\), the generic point would give a \(K_0\)-rational point of
\(C_{\rm gen}\); because \(D_i\) meets \(\widetilde U\), that point is
not in \(\mathcal T\), contrary to
Theorem~\ref{thm:no-nontrivial-generic-sections}.  Every horizontal
curve contributing to \(\mathcal K_{\rm all}\) therefore has degree
at least \(2\).  The image of its rational points is thin by the
definition of a thin set.  A finite union of these images and the
finitely many vertical parameters is again thin.

Now suppose that there are infinitely many primitive all-square
triangles up to edge permutation.  The finite-orbit correspondence
above gives infinitely many points of
\(U(\Q)=\widetilde U(\Q)\).  Hence one
of the finitely many curves supplied by weak Bombieri--Lang contains
infinitely many rational points.  It cannot be vertical: every
vertical curve meeting \(\widetilde U\) has \(b\ne0\), and its smooth
projective completion is the genus-three curve \(C_{a,b}\), which has
only finitely many rational points by Faltings's theorem.  It is
therefore a multisection, and its normalization
\(D^\nu\) still has infinitely many rational points.  Faltings's
theorem forces \(g(D^\nu)\le1\).  In genus \(0\) there is a rational
point; in genus \(1\), a rational point identifies \(D^\nu\) with its
Jacobian, and infinitude forces positive Mordell--Weil rank.  The
degree-one argument above gives \(\deg(D/\PP^1)\ge2\).
\end{proof}

\begin{remark}
Up to the sign changes of the square roots, the nonboundary isosceles
locus is precisely the rank-zero genus-one bisection of
Proposition~\ref{prop:isosceles-genus-one-bisection} and its
edge-permuted counterparts.
Corollary~\ref{cor:no-right-isosceles-all-square} shows that none of
these components contributes an all-square triangle.  Consequently,
weak Bombieri--Lang for \(\widetilde{\mathscr H}\), together with the
nonexistence of the multisections asked for in Open Problem~1, would
force finiteness in Open Problem~2.  Weak Bombieri--Lang alone yields
only thinness, not finiteness.
\end{remark}

\section{Arithmetic of the recorded all-square fibers}\label{sec:examples}

The following table gives one normalized representative for each
recorded primitive all-square triangle:
{\renewcommand{\arraystretch}{1.15}
\[
\begin{array}{@{}c@{\qquad}c@{}}
\toprule
\text{triangle}&(k,z,q)\\
\midrule
(1853^2,4380^2,4427^2)
&
\left(\dfrac{1009}{9649},
      \dfrac{1853}{4427},
      \dfrac{4380}{4427}\right)
\\
(11789^2,68104^2,68595^2)
&
\left(\dfrac{644437}{2437269},
      \dfrac{11789}{68595},
      \dfrac{68104}{68595}\right)
\\
\bottomrule
\end{array}
\]
}
Their respective areas are \(32918611718880\) and
\(284239560530875680\); these values are also recorded in
OEIS~A318575 \cite{OEISSquareAreas}.
Write \((k_i,z_i,q_i)\) for the entries in row \(i\), and put
\[
\begin{gathered}
 a_i=\frac{2k_i}{k_i^2+1},\qquad
 b_i=\frac{k_i^2-1}{k_i^2+1},\qquad
 v_i=z_i^2-a_i,\\
 R_i=\bigl(2(q_i^2+v_i),\,4q_i(q_i^2+v_i)\bigr)\in E_{b_i}(\Q),\\
 Q_i=\bigl(2(1+a_i),\,4(1+a_i)\bigr)\in E_{b_i}(\Q).
\end{gathered}
\]
Let \((U_i,V_i)=\phi_3(z_i,q_i)\) and
\(P_i=(2U_i,2V_i)\in F'_{a_i}(\Q)\).  Direct substitution verifies the
quartic equation and shows that all the displayed points lie on the
stated curves.

\begin{proposition}[Certified rank jumps on the recorded fibers]
\label{prop:recorded-rank-jumps}
For each \(i=1,2\), the points \(R_i,Q_i\in E_{b_i}(\Q)\) are linearly
independent and \(P_i\in F'_{a_i}(\Q)\) has infinite order.
Consequently, on both recorded all-square fibers,
\[
 \rank E_{b_i}(\Q)\ge2,\qquad
 \rank F'_{a_i}(\Q)\ge1,\qquad
 \rank\Jac(C_{a_i,b_i})(\Q)\ge5.
\]
\end{proposition}

\begin{proof}
For \(E_b:y^2=x^3+4b^2x\), the descent map associated with the
two-isogeny having kernel \(\{O,(0,0)\}\) is the homomorphism
\(\alpha_b:E_b(\Q)\to\Q^\times/\Q^{\times2}\), with
\(\alpha_b((x,y))=x\,\Q^{\times2}\) for \(x\ne0\) and
\(\alpha_b(O)=\alpha_b((0,0))=1\)
\cite[Chapter~X, Proposition~4.9]{Silverman}.  Exact factorization gives
{\renewcommand{\arraystretch}{1.15}
\[
\begin{array}{c|cc}
i&\alpha_{b_i}(R_i)&\alpha_{b_i}(Q_i)\\
\hline
1&85=5\cdot17&3961=17\cdot233\\
2&17026=2\cdot8513&22865=5\cdot17\cdot269.
\end{array}
\]
}
Writing \(k_i=m_i/n_i\), the identities
\[
\begin{gathered}
 \dfrac{x(R_1)}{85}=\left(\dfrac{5256}{35207}\right)^2,\quad
 \dfrac{x(R_2)}{17026}=\left(\dfrac{372136}{47568615}\right)^2,\\
 2(m_1^2+n_1^2)=3961\cdot218^2,\quad
 2(m_2^2+n_2^2)=22865\cdot23578^2
\end{gathered}
\]
certify the table, since
\(x(Q_i)=2(m_i+n_i)^2/(m_i^2+n_i^2)\).  The two square classes in each
row are linearly independent in \(\Q^\times/\Q^{\times2}\).

The only nonzero rational two-torsion point on \(E_{b_i}\) is
\(T=(0,0)\).  The formula
\(x(2P)=(x(P)^2-4b_i^2)^2/(4y(P)^2)\) shows that rational
four-torsion would force \(b_i\) or \(-b_i\) to be a rational square.
This is impossible because \(b_i<0\), while
\(\ord_5(-b_1)=1\) and \(\ord_5(-b_2)=-1\).  Hence
\(E_{b_i}(\Q)[2^\infty]=\{O,T\}\).  If
\(nR_i+mQ_i=O\), applying \(\alpha_{b_i}\) makes \(n,m\) even.
Writing \(n=2n_1\), \(m=2m_1\) gives
\(S=n_1R_i+m_1Q_i\in\{O,T\}\).  Since \(\alpha_{b_i}(S)=1\),
independence of the two square classes makes \(n_1,m_1\) even.
More generally, suppose that \(2^r\mid n,m\), and put
\[
 S_r=\frac{n}{2^r}R_i+\frac{m}{2^r}Q_i.
\]
Then \(2^rS_r=O\), so
\(S_r\in E_{b_i}(\Q)[2^\infty]=\{O,T\}\).
Hence \(\alpha_{b_i}(S_r)=1\), and independence of the two square
classes implies \(2^{r+1}\mid n,m\).  Induction gives
\(2^r\mid n,m\) for every \(r\), and therefore \(n=m=0\).

Since \(\Delta(F'_{a_i})=2^{12}b_i^4\), exact reduction at \(43\)
gives
{\renewcommand{\arraystretch}{1.0}
\[
\begin{array}{c|c|c|c|c}
i&\overline a_i&\overline P_i&
\Delta(F'_{a_i})\pmod{43}&13\overline P_i\\
\hline
1&35&(5,22)&14&(41,0)\\
2&20&(26,3)&36&(40,0).
\end{array}
\]
}
Thus both curves have good reduction; the last column consists of
nonzero \(2\)-torsion points, and the displayed \(y\)-coordinates of
\(\overline P_i\) are nonzero.  Hence each \(\overline P_i\) has exact
order \(26\).  If \(P_i\) were torsion over \(\Q\), its order would be
divisible by \(13\), contrary to Mazur's torsion theorem
\cite{Mazur}.  Hence \(P_i\) has infinite order.  Finally, the
\(\Q\)-isogeny
\(\Jac(C_{a_i,b_i})\sim_{\Q}E_{b_i}^2\times F'_{a_i}\)
gives the asserted rank bound.
\end{proof}

Thus each specialization has Jacobian rank at least \(5\), hence at
least three more than the generic arithmetic rank \(2\), so the
classical Chabauty condition \eqref{eq:chabauty-rank-condition} fails.
These rank jumps are not used elsewhere.

\paragraph{A quadratic-Chabauty range.}

For a smooth rational fiber put \(J=\Jac(C_{a,b})\).  On \(E_b^2\),
the two factor classes and the diagonal have intersection matrix of
determinant \(2\).  Together with the point class on \(F_a\), their
pullbacks under the \(\Q\)-isogeny \(J\to E_b^2\times F_a\) give four
independent Galois-invariant divisor classes; hence
\(\rho_{\Q}(J):=\rank\operatorname{NS}(J_{\overline{\Q}})
^{\operatorname{Gal}(\overline{\Q}/\Q)}\ge4\).
Hence \(2r_b+r_a<6\) implies the Balakrishnan--Dogra numerical
dimension condition for quadratic Chabauty; its application remains
subject to the standard local and \(p\)-adic-height hypotheses, which
are not asserted here \cite{BalakrishnanDogra}.

\section{Open problems}
\label{sec:questions}

\begingroup
\setlength{\topsep}{0pt}
\setlength{\partopsep}{0pt}
\setlength{\itemsep}{0pt}
\setlength{\parsep}{0pt}
\setlength{\parskip}{0pt}
\begin{enumerate}
\item Apart from the rank-zero isosceles genus-one bisection of
Proposition~\ref{prop:isosceles-genus-one-bisection} and its
edge-permuted counterparts, classify the geometrically integral
multisections of degree at most \(4\) whose normalization is either
rational with a \(\Q\)-point or genus one with a \(\Q\)-point and
positive Jacobian rank.  Do any non-isosceles examples exist, or pass
through a recorded all-square point?
\item Are there finitely or infinitely many unmarked rational
all-square similarity classes?  In particular, does a third primitive
integral all-square Heron triangle exist, and can the case
\(3\nmid XYZ\) in
Proposition~\ref{prop:all-square-local-obstructions} be excluded?
\item For a fixed \(k\in\mathbf Q\setminus\{0,\pm1\}\), are there
infinitely many marked
classes satisfying \eqref{eq:primitive-criterion}?  More generally,
how do \(\rank E_b(\Q)\) and \(\rank F_a(\Q)\) vary with \(k\), and
what are the exact ranks and rational points on the two recorded
all-square fibers?
\end{enumerate}
\endgroup
Theorem~\ref{thm:no-nontrivial-generic-sections} excludes nontrivial
rational sections, but not higher-degree multisections or nontrivial
rational points occurring only on individual fibers;
Corollary~\ref{cor:conditional-all-square-thinness} is conditional and
does not decide Open Problem~2.

\end{document}